\documentclass[prerint,11pt]{article}

\usepackage{amssymb,amsfonts,amsmath,amsthm,amscd,dsfont,mathrsfs}
\usepackage{graphicx,float,psfrag,epsfig,color}
\usepackage{hyperref}

\newtheorem{claim}{Claim}
\newtheorem{lemma}{Lemma}

\newtheorem{theorem}{Theorem}
\newtheorem{proposition}[claim]{Proposition}

\newtheorem{definition}[claim]{Definition}

\newtheorem{remark}{Remark}

\definecolor{Red}{rgb}{1,0,0}
\definecolor{Blue}{rgb}{0,0,1}
\definecolor{Olive}{rgb}{0.41,0.55,0.13}
\definecolor{Green}{rgb}{0,1,0}
\definecolor{MGreen}{rgb}{0,0.8,0}
\definecolor{DGreen}{rgb}{0,0.55,0}
\definecolor{Yellow}{rgb}{1,1,0}
\definecolor{Cyan}{rgb}{0,1,1}
\definecolor{Magenta}{rgb}{1,0,1}
\definecolor{Orange}{rgb}{1,.5,0}
\definecolor{Violet}{rgb}{.5,0,.5}
\definecolor{Purple}{rgb}{.75,0,.25}
\definecolor{Brown}{rgb}{.75,.5,.25}
\definecolor{Grey}{rgb}{.5,.5,.5}
\definecolor{Pink}{rgb}{1,0,1}
\definecolor{DBrown}{rgb}{.5,.34,.16}

\definecolor{Black}{rgb}{0,0,0}

\definecolor{White}{rgb}{1,1,1}
\def\white{\color{White}}

\def\sT{{\sf T}}
\def\id{{\rm I}}

\def\tq{\tilde{q}}
\def\Space{{\cal V}_{q,N}}
\def\nSpace{{\cal V}_{q,n}}
\def\mSpace{{\cal V}_{q,m}}
\def\cF{{\cal F}}

\def\reals{{\mathbb R}}
\def\naturals{{\mathbb N}}
\def\<{\langle}
\def\>{\rangle}
\def\E{{\mathbb E}}
\def\tA{\tilde{A}}
\def\normal{{\sf N}}

\def\iid{\textrm{i.i.d. }}

\def\he{\widehat{e}}
\def\hh{\widehat{h}}

\def\vv{\mathbf{v}}
\def\vw{\mathbf{w}}
\def\vx{\mathbf{x}}
\def\vm{\mathbf{m}}

\def\Ens{{\cal M}}
\def\Rows{{\sf R}}
\def\Cols{{\sf C}}
\def\limsup{\rm{lim \, sup}}
\def\liminf{\rm{lim \, inf}}
\def\gr{{\sf g}}
\def\mmse{{\sf mmse}}

\def\Lr{{L_r}}
\def\Lc{{L_c}}
\def\tQ{\widetilde{Q}}
\def\teta{\tilde{\eta}}
\def\matB{\mathscr{B}}

\def\no{\nonumber}
\def\uRenyi{\overline{d}}

\def\Om{{\bf B}}

\def\v1{\mathbf{1}}
\def\vu{\mathbf{u}}
\def\vgamma{\mathbf{\gamma}}
\def\vy{\mathbf{y}}
\def\vz{\mathbf{z}}
\def\va{\mathbf{a}}

\def\vw{\mathbf{w}}

\def\vZ{\mathbf{Z}}

\newcommand{\pl}{\parallel}
\def\va{\vec{\alpha}}

\newcommand{\mc}{\mathcal}
\newcommand{\deq}{\stackrel{\text{\rm d}}{=}}
\newcommand{\asequal}{\stackrel{\almostsurely}{=}}
\newcommand{\almostsurely}{{\rm a.s.}}
\newcommand{\order}{\vec{o}}

\newcommand{\indicator}{\mathbb{I}}
\newcommand{\cov}{\textrm{\rm Cov}}
\newcommand{\pty}{\mc}
\newcommand{\psL}{\mathrm{PL}}

\newcommand{\empr}{\hat{p}}
\def\E{{\mathbb E}}

\def\cF{{\cal F}}

\def\ons{{\sf b}}
\def\pons{{\sf d}}
\def\Ons{{\sf B}}
\def\POns{{\sf D}}

\def\tx{\tilde{x}}
\def\tr{\tilde{r}}

\def\de{{\rm d}}
\def\eps{{\varepsilon}}

\def\Var{{\rm Var}}

\def\hSigma{\widehat{\Sigma}}

\newcommand{\sigal}{\mathfrak}
\newcommand{\tM}{\tilde{M}}
\newcommand{\tm}{\tilde{m}}
\newcommand{\f}{\frac}
\def\lbq{\rho}
\def\nphi{\nabla \varphi}

\newcommand{\BEAS}{\begin{eqnarray*}}
\newcommand{\EEAS}{\end{eqnarray*}}
\newcommand{\BEA}{\begin{eqnarray}}
\newcommand{\EEA}{\end{eqnarray}}
\begin{document}
\title{State Evolution for General Approximate Message Passing
  Algorithms, with Applications to Spatial Coupling}

\author{Adel Javanmard\footnote{Department of Electrical Engineering, Stanford
    University}  \quad and \quad Andrea Montanari
            \footnote{Department of Electrical Engineering and Department of Statistics, Stanford University}}

\maketitle

%
%
\begin{abstract}
We consider a class of approximated message passing (AMP)
algorithms and characterize their high-dimensional
behavior in terms of a suitable state evolution recursion. Our proof
applies to Gaussian matrices with independent but not necessarily
identically distributed entries. It covers --in particular-- the
analysis of generalized AMP, introduced by Rangan, and 
of AMP reconstruction in compressed sensing with spatially coupled
sensing matrices.

The proof technique builds on the one of \cite{BM-MPCS-2011}, while 
simplifying and generalizing several steps.
\end{abstract}
%
%
\section{Introduction}

Approximate message passing (AMP) algorithms \cite{DMM09} apply ideas from graphical
models (belief propagation \cite{Pearl}) and statistical physics (mean field or TAP
equations \cite{SpinGlass,MezardMontanari}) to statistical
estimation. In particular AMP applies to problems that  do not admit a sparse  graphical model description.
An AMP algorithm takes the form
\begin{eqnarray}
u^{t} &= & A\, f(v^t;t)  - \ons_t \, g(u^{t-1};t-1)\, ,\label{eq:FirstAMP1}\\
v^{t+1} & = &A^{\sT}\, g(u^t;t) - \pons_t\,  f(v^t;t)\, ,\label{eq:FirstAMP2}
\end{eqnarray}
with $t\in \naturals$ being the iteration number.
Here $v^t\in\reals^n$, $u^t\in\reals^m$ are vectors that describe the
algorithm's state, $f(\,\cdot\,;t):\reals^n\to\reals^n$ and
$g(\,\cdot\,;t):\reals^m\to\reals^m$ are sequences of functions that
can be computed efficiently and $\ons^t$, $\pons^t$ are scalars that
also can be computed given the current state. Finally
$A\in\reals^{m\times n}$ is a matrix that is given as part of the data of the
estimation problem. 

One domain in which AMP  finds application is  the ubiquitous problem
of estimating an unknown signal $x\in\reals^n$ from noisy linear observations:
\begin{eqnarray}
y = A\, x+w\, .\label{eq:LinearModel}
\end{eqnarray}
Here $A\in\reals^{m\times n}$ is a known sensing matrix and $w\in\reals^m$
is a noise vector with \iid components with $\E w_i=0$,
$\E\{w_i^2\}=\sigma^2$. In \cite{DMM09} a class of  AMP algorithms was
developed for this problem in the compressed sensing setting in which
$x$ is sparse and $m<n$. Several generalization --for instance to
signals with small total variation-- were developed in
\cite{DonohoJohnstoneMontanari}, which also provides a more complete
list of references. All of these generalizations can be recast on the
form of Eqs.~(\ref{eq:FirstAMP1}), (\ref{eq:FirstAMP2}) for suitable choices of the
functions $f(\,\cdot\,;t)$ and $g(\,\cdot\,;t)$.

A striking property of AMP algorithms is that their high-dimensional
behavior admits an \emph{exact description}. Simplifying, for a
broad range of random matrices $A$, the vectors
$u^t$, $v^t$ have asymptotically \iid Gaussian entries in the limit $n,m\to\infty$ at
$t$ fixed (see next section for a formal statement). The variance of
$u^t_i$, $v^t_i$ can be computed through a one-dimensional recursion
termed \emph{state evolution}, because of its analogy with density
evolution in coding theory \cite{RiU08}. The predictions of state
evolution were tested numerically in several papers, see e.g.
\cite{DMM09,DMM-NSPT-11,DonohoJohnstoneMontanari,SchniterTurbo,RanganQuantized,krzakala2012probabilistic,som2012compressive,JavanmardMon12}.
In \cite{BM-MPCS-2011} it was proved that state evolution does indeed
hold if $A$ has \iid Gaussian entries and the functions
$f(\,\cdot\,;t)$ and $g(\,\cdot\,;t)$ are Lipschitz continuous and
separable\footnote{Throughout the paper we say that
  $h:\reals^k\to\reals^k$ is separable if $h(x_1,x_2,\dots,x_k) = (h_1(x_1),h_2(x_2),\dots,h_k(x_k))$.}. This result was extended in
\cite{BM-Universality}  to matrices $A$ that have independent
non-Gaussian entries, under the assumption that functions
$f(\,\cdot\,;t)$ and $g(\,\cdot\,;t)$ are separable polynomials.
On the basis of these results, it is natural to conjecture that state
evolution holds  for matrices with general independent entries, whenever 
$f(\,\cdot\,;t)$ and $g(\,\cdot\,;t)$ are separable and locally
Lipschitz with polynomial growth. This conjecture is still
open. 

In this paper we focus on Gaussian matrices and consider a different
type of generalization that was motivated by the following recent
developments.
\begin{description}
\item[Generalized AMP.] In \cite{RanganGAMP}, Rangan proposed a class of
generalized message passing algorithms (G-AMP) which found several  interesting
applications, see \cite{fletcher2011neural,kamilov2012one}. In
particular, generalized AMP allows to tackle  nonlinear
estimation problems wherein $x\in\reals^n$ is to be estimated from
observations $Y = (Y_1,\dots,Y_m)$. Observations are conditionally
independent given $A$ and $x$, with $Y_i$ distributed according
to a model  $p(\,\cdot\,|\xi_i)$ with $\xi_i =
(A x)_i$. Considering for simplicity the case in which
$p(\,\cdot\,|\xi_i)$ has a density (denoted again by $p$), the joint
density of $Y = (Y_1,\dots,Y_m)$ is therefore
\begin{eqnarray}
p_Y(y|A,x) = \prod_{i=1}^m p\big(y_i|(A x)_i\big)\, .
\end{eqnarray}
In information theory parlance, the vector $(A
x)$ is passed through a memoryless channel with transition probability
$p(\,\cdot\,|\,\cdot\,)$. 
From a statistics point of view, this corresponds to estimation of a
generalized linear model~\cite{Neld:Wedd:1972,McCu:Neld:1989}.
The linear model (\ref{eq:LinearModel}) is recovered as the special case in
which the channel is Gaussian or --more generally-- the noise is
purely additive.
Rangan conjectured that suitable state evolution equations hold for
G-AMP algorithms as well, without however providing a formal proof.
\item[Spatial coupling.] In a separate line of work, Donoho and the present authors \cite{DJM} applied AMP
to compressed sensing reconstruction with spatially coupled sensing
matrices. This type of sensing matrices were developed in \cite{KrzakalaEtAl} (see also \cite{KudekarPfister} for
earlier work in this direction), who demonstrated heuristically the
power of this approach. 
A mathematical analysis requires extending state evolution
to matrices with independent centered Gaussian entries, although with
non-identical variances (heteroscedastic entries, in the statistics
terminology). More precisely, for $A\in\reals^{m\times n}$ we assume 
that the row index set $[m] = \{1,\dots,m\}$ is partitioned into $q$
groups, and that the same holds for the column index set $[n] =
\{1,\dots,n\}$.
Then the entries $A_{ij}$ are independent Gaussian with mean
$\E\{A_{ij}\}=0$  and variance $\E\{A_{ij}^2\}$ depending on the
group to which $i$ and $j$ belong.
Spatially coupled sensing matrices correspond to a special
band-diagonal structure of the block variances.

 A rigorous analysis of the implications of state evolution for
 spatially coupled matrices can be found in  \cite{DJM}.
In particular, \cite{DJM} studied a class of spatially coupled
matrices, and proved that AMP reconstruction achieves the
information-theoretic limit stated in~\cite{WuVerdu}. 
More specifically, for sequences of spatially coupled matrices $A \in
\reals^{m\times n}$ with asymptotic undersampling rate 
$\delta = \lim_{n\to \infty} m/n$, AMP reconstructs the signal with
high probability, provided $\delta > \uRenyi(p_X)$, 
where $\uRenyi(p_X)$ denotes the (upper) R\'enyi information dimension
of $p_X$~\cite{Renyi}. Further, AMP reconstruction is robust to noise.
\item[Robust regression.]  Bean, Bickel, El
  Karoui and Yu \cite{bean2012optimal} recently considered the problem
  of estimating the unknown vector $x$ in the linear model
  (\ref{eq:LinearModel}) using robust regression. They developed exact
  asymptotic expressions for the risk that are
  analogous to the one proved in \cite{BayatiMontanariLASSO} for the
  Lasso. The results of \cite{bean2012optimal} are, on the other hand,
  based on an heuristic derivation.

  The proof in \cite{BayatiMontanariLASSO} was based on the
  state evolution analysis of a suitable AMP algorithm whose fixed
  points coincide with the Lasso optima. This is suggestive of a
  possible approach for proving the results of  \cite{bean2012optimal}:
  define a suitable AMP algorithm for solving the robust
  regression problem, and analyze it through state evolution.
  Indeed a comparison of the formulae in \cite{bean2012optimal} with
  the state evolution formulae in \cite{RanganGAMP} appears encouraging.
\end{description}
In this paper we establish a rigorous generalization of state
evolution that covers all of the above developments. Applications to
generalized AMP  are already discussed in \cite{RanganGAMP}, and
applications to spatially coupled sensing matrices can be found in
\cite{DJM} and Section \ref{sec:SpatialCoupling}. Finally,
applications to robust regression are left for future study.

Remarkably, all of the above applications can be derived by treating
the following generalization of  the iteration (\ref{eq:FirstAMP1}),
(\ref{eq:FirstAMP2}). (A formal definition is given in the next section.)
\begin{enumerate}
\item The vectors $u^t\in\reals^{m}$, $v^t\in\reals^n$ are replaced by
  matrices $u^t\in\reals^{m\times q}$, $v^t\in\reals^{n\times q}$,
  with $q$ kept fixed as $m,n\to\infty$.
\item The functions $f,g$ appearing in Eqs.~(\ref{eq:FirstAMP1}),
  (\ref{eq:FirstAMP2}) are now mappings $f(\,\cdot\,;t):\reals^{n\times q}\to \reals^{n\times q}$,
 $g(\,\cdot\,;t):\reals^{m\times q}\to \reals^{m\times q}$ that are
 separable across rows (e.g. the $i$-th row of $f(v;t)$ only depends
 on the $i$-th row on $v$). Correspondingly, the product  $A
 f(v^t;t)$ has to be  interpreted as a matrix multiplication.
\item The memory terms are modified with $\ons_t$, $\pons_t$ replaced
  by $q\times q$ matrices. More specifically, $\ons_t\, g(u^{t-1};t-1)$ and $\pons_t\, f(v^t;t)$
  are respectively replaced by $g(u^{t-1};t-1)\, \Om_t^\sT$, $f(v^{t};t)\, {\bf D}_t^\sT$,
  with $\Om_t, {\bf D}_t \in \reals^{q\times q}$.
\end{enumerate}

Our proof uses the technique of \cite{BM-MPCS-2011}, which in turns
build on an idea first introduced by Bolthausen
\cite{bolthausen2012iterative}. A convenient simplification with
respect to \cite{BM-MPCS-2011} consists in studying a recursion in
which the rectangular matrix $A$ is replaced by a symmetric matrix,
and the algorithm state is described by a single vector. 

In section \ref{sec:MainResult} we put forward formal definitions and
state our main result for the case of symmetric matrices. In section
\ref{sec:SpatialCoupling} we
show how the case of rectangular matrices can be reduced to the
symmetric one. We also show how our result applies to the case of
compressed sensing reconstruction with spatially coupled
matrices. Finally, we prove our main result in Section \ref{sec:Proof}.

\section{Main result}
\label{sec:MainResult}

We will view AMP as operating on the vector space
$\Space \equiv (\reals^q)^N\simeq \reals^{N\times q}$. Given a
vector $x\in \Space$, we shall most often regard  it as an
$N$-vector with entries in $\reals^q$, namely $x =
(\vx_1,\dots,\vx_N)$, with $\vx_i\in\reals^q$.
Components of
$\vx_i\in\reals^q$ will be indicated as $(\vx_i(1),\dotsc,\vx_i(q))\equiv \vx_i$.
For $x \in \Space$, we define its norm by $\|x\| = \left(\sum_{i=1}^N \|\vx_i\|^2\right)^{1/2}$.

Given a matrix $A\in\reals^{N\times N}$, we let it act on
$\Space$ in the natural way, namely for $v', v\in \Space$ we let $v'=Av$ be given by
$\vv'_i = \sum_{j=1}^NA_{ij}\vv_j$ for all $i\in [N]$.
Here and below $[N]\equiv \{1,\dots,N\}$ is the set of first $N$ integers.
In other words we identify $A$ with the
Kronecker product $A\otimes \id_{q\times q}$.
\begin{definition}\label{def:GeneralDef}
A symmetric \emph{AMP instance} is a triple $(A,\cF,x^0)$ where:
\begin{enumerate}
\item  $A = G + G^\sT$, where $G \in\reals^{N\times N}$ has i.i.d. entries $G_{ij}\sim \normal(0,(2N)^{-1})$.
\item $\cF = \{f^k:  k\in [N]\}$ is a collection of mappings
$f^k:\reals^q\times\naturals\to\reals^q$, $(\vx,t)\mapsto f^k(\vx,t)$  that are locally Lipschitz in
their first argument (and hence almost everywhere differentiable);
\item  $x^0\in \Space$ is an initial condition.
\end{enumerate}
Given  $\cF = \{f^k:  k\in [N]\}$, we define  $f(\,\cdot\,;t):\Space\to \Space$ by letting $v'=f(v;t)$ be
given by $\vv'_i = f^i(\vv_i;t)$ for all $i\in [N]$.
\end{definition}

\begin{definition}
The \emph{approximate message
  passing orbit} corresponding to the instance $(A,\cF,x^0)$ is the sequence of vectors $\{x^t\}_{t\ge 0}$, $x^t\in\Space$
defined as follows, for $t\ge 0$,
\begin{eqnarray}
x^{t+1} = A\, f(x^t;t) - \Ons_t \, f(x^{t-1};t-1)\, . \label{eq:AMPGeneralDef}
\end{eqnarray}
Here $\Ons_t: \Space\to\Space$ is the linear operator defined by
letting, for $v' = \Ons_t v$,
\begin{eqnarray}
\vv'_i = \frac{1}{N}\left(\sum_{j\in [N]}\frac{\partial f^j}{\partial \vx}(\vx^t_j,t)\right) \vv_i\, ,\label{eq:OnsT_Def}
\end{eqnarray}
with $\frac{\partial f^j}{\partial \vx}$ denoting the Jacobian matrix
of $f^j(\,\cdot\,;t):\reals^q\to\reals^q$.
\end{definition}
%
%
%
\subsection{State evolution}
\label{sec:StateEvolutionResults}

In order to establish the behavior of the
sequence $\{x^t\}_{t\ge 0}$ in the high
dimensional limit, we need to
consider a sequence of AMP instances $\{A(N), \cF_N, x^{0,N}\}_{N \ge
  0}$ indexed by the dimension $N$.

\begin{definition}\label{def:Converging}
We say that the sequence of AMP instances $\{(A(N),\cF_N,x^{0,N})\}_{N\ge 0}$
is \emph{converging} if there exists:
$(i)$ An  integer $q$;
$(ii)$ A function $g:\reals^q\times\reals^{q}\times[q]\times
\naturals\to\reals^q$ with $g(\vx,\vy,a,t) = (g_1(\vx,\vy,a,t),\cdots,
g_q(\vx,\vy,a,t))$, 
such that, for each $r\in [q]$, $a \in [q]$, $t \in \naturals$, 
$g_r(\cdots,a,t)$ is Lipschitz continuous;
$(iii)$ $q$ probability measures $P_1$, \dots, $P_q$ on $\reals^{q}$;
$(iv)$  For each $N$, a finite partition $C^N_1\cup C^N_2\cup \dots \cup
C^N_q=[N]$;
$(v)$ $q$ positive definite matrices $\hSigma^1_1,\dotsc, \hSigma^1_q \in \reals^{q \times q}$, such that the following happens;
\begin{enumerate}
\item For each $a\in [q]$, we have $\lim_{N\to\infty} |C^N_a|/N =
  c_a\in (0,1)$.
\item For each $N\ge 0$, each $a\in [q]$ and each $i\in C_a^N$, we
  have $f^i(\vx,t) = g(\vx,\vy_i,a,t)$. Further,  the empirical distribution of $\{\vy_i\}_{i\in C^N_a}$, denoted by $\hat{P}_a$, 
  converges weakly to $P_a$. 
\item For each  $a\in [q]$, in probability,
\begin{eqnarray}
\lim_{N\to\infty} \frac{1}{|C_a^N|}\sum_{i\in C_a^N}g\Big(\vx^0_i,\vy_i,a,0\Big) g\Big(\vx^0_i,\vy_i,a,0\Big)^{\sT}
= \hSigma^0_a \, .\label{eq:InitialSE}
\end{eqnarray}
\end{enumerate}
\end{definition}
\begin{remark}
An apparent generalization of the above definition would require the
partition to be $C^N_1\cup C^N_2\cup \dots \cup
C^N_{q'}=[N]$, while $x^t\in\Space$, with $q\neq q'$. It is easy to
see that there is no loss of generality in assuming $q=q'$ as we do in
our definition. Indeed the case $q'<q$ can be reduced to our
setting by refining the partition arbitrarily, and $q'>q$ by adding
dummy coordinates to to the variables $\vx_i$. 
\end{remark}
\begin{remark}
The function $f^i(\,\cdot\,,\,\cdot\,)$ depends implicitly on $\vy_i$.
However, the $\vy_i$'s do not change across iterations and so we do not show
this dependence explicitly in our notation.
\end{remark}

Our next result establishes that the low-dimensional marginals of $\{x^t\}$ are
asymptotically Gaussian.
\emph{State evolution}  characterizes the covariance of these
marginals. For each $t\ge 1$, state evolution defines a 
positive semidefinite matrix $\Sigma^t \in\reals^{q\times q}$. This is obtained by letting, for each $t\ge 1$
\begin{eqnarray}
\Sigma^{t} &=& \sum_{b=1}^q c_b\,\, \hSigma^{t-1}_b\, ,\label{eq:GeneralSE1}\\
\hSigma^t_a &=& \E\left\{  g(Z^t_a,Y_a,a,t) g(Z^t_a,Y_a,a,t)^{\sT}\right\}\,\label{eq:GeneralSE2} ,
\end{eqnarray}
for all $a\in [q]$.
Here $Y_a\sim P_a$, $Z_a^t\sim \normal\left(0, \Sigma^t \right)$ and
$Y_a$ and $Z_a^t$ are independent.

For $k\ge1$ we say a function $\phi: \reals^m \to\reals$ is \emph{pseudo-Lipschitz}
of order $k$ and denote it by $\phi\in \psL(k)$ if there exists a constant
$L>0$ such that, for all $x,y\in\reals^m$:
\begin{eqnarray}
|\phi(x)-\phi(y)|\le L(1+\|x\|^{k-1}+\|y\|^{k-1})\, \|x-y\|\, .
\end{eqnarray}
Notice that if $\phi\in \psL(k)$, then there exists a constant $L'$ such that for all $x\in\reals^m$: $|\phi(x)|\leq L'(1+\|x\|^k)$.

\begin{theorem}\label{thm:SE}
Let $(A(N),\cF_N,x^0)_{N\ge 0}$ be a converging sequence of AMP
instances, and denote by $\{x^t\}_{t\ge 0}$ the corresponding AMP
sequence. Suppose further that  $\E_{P_a}(\|Y_a\|^{2k-2})$ is bounded,
and $\E_{\hat{P}_a}(\|Y_a\|^{2k-2}) \to \E_{P_a}(\|Y_a\|^{2k-2})$ as
$N \to \infty$,
for some $k\ge 2$. Then for all $t\ge 1$, each $a\in [q]$,  and any pseudo-Lipschitz function
$\psi:\reals^q\times\reals^q\to\reals$ of order $k$, we
have, almost surely,
\begin{eqnarray}
\lim_{N\to\infty} \frac{1}{|C_a^N|}\sum_{j\in C^N_a} \psi(\vx^t_j,\vy_j) =
\E\{\psi(Z_a^t,Y_a)\}\, ,
\end{eqnarray}
where $Z_a^t\sim \normal(0,\Sigma^t)$ is independent of $Y_a\sim P_a$.
\end{theorem}

%
%
\section{AMP for rectangular and spatially-coupled matrices}
\label{sec:SpatialCoupling}

In this section we develop two applications of our main theorem:
\begin{enumerate}
\item We show that AMP iterations with  $A$ a rectangular matrix, see
  e.g. Eqs.~(\ref{eq:FirstAMP1}), (\ref{eq:FirstAMP2}), can
  be recast in the form of an iteration with a symmetric matrix $A$
  and are therefore covered by Theorem \ref{thm:SE}. This construction
  is provided in Section \ref{sec:ProofLemmaSE} (below Proposition
  \ref{pro:spatial-special}).
\item We apply the general Theorem \ref{thm:SE} to AMP
  reconstruction in compressed sensing with spatially
  coupled matrices. In \cite{DJM}, it was proved that, conditionally
  to a state evolution lemma, this approach achieves the
  information-theoretic limits of compressed sensing set forth in \cite{WuVerdu}.
  Here we show that our main result Theorem \ref{thm:SE}
  implies the state evolution lemma  (Lemma~$4.1$ in~\cite{DJM}).
\end{enumerate}

\subsection{General matrix ensemble}
\label{sec:MatrixEnsemble}

We begin by describing a more general matrix ensemble that encompasses 
spatially coupled matrices, 
and will be denoted by $\Ens(W,m_0,n_0)$. The ensemble depends on two
integers $m_0,n_0\in\naturals$, and on a matrix with non-negative
entries $W\in \reals_+^{\Rows\times \Cols}$,
whose rows and columns are indexed by the finite sets $\Rows$, $\Cols$
(respectively `rows' and `columns').
The matrix is \emph{roughly row-stochastic}, i.e. 
\begin{eqnarray}\label{eqn:almost_row_stochastic}
\frac{1}{2}\le\sum_{c\in\Cols}W_{r,c} \le 2\, ,\;\;\;\;\;\;\; \mbox{for all
}r\in\Rows\, .
\end{eqnarray}
We will let $|\Rows|\equiv L_r$ and $|\Cols| \equiv L_c$ denote the 
matrix dimensions.
The ensemble parameters are related to the
sensing matrix dimensions by $n= n_0L_c$ and $m=m_0L_r$.

In order to describe a random matrix $A\sim\Ens(W,m_0,n_0)$ from this
ensemble, partition the column and row indices of $A$ in --respectively--
$L_c$ and $L_r$ groups of equal
size. Explicitly
\begin{align*}
[n] &= \cup_{s\in\Cols}C_s\, ,\;\;\;\;  |C_s|= n_0\, ,\\
[m] &= \cup_{r\in\Rows}R_r\, ,\;\;\;\;  |R_r|=m_0\, .
\end{align*}
Further, if $i\in R_r$ or $j\in C_s$ we
will write, respectively,   $r= \gr(i)$ or $s= \gr(j)$. In other
words $\gr(\,\cdot\, )$ is the operator determining the group index of
a given row or column.

With this notation we have the following concise definition of the
ensemble.
\begin{definition}
A random sensing matrix  $A$ is distributed according to the ensemble
$\Ens(W,m_0,n_0)$ (and we write $A\sim \Ens(W,m_0,n_0)$) if the entries
$\{A_{ij}, \;\; i\in [m], j\in [n]\}$ are independent  Gaussian random
variables with 
\begin{eqnarray}\label{eqn:A_W}
A_{ij}\sim \normal\Big(0,\frac{1}{m_0}\, W_{\gr(i),\gr(j)}\Big)\, .
\end{eqnarray}
\end{definition}
See Fig.~\ref{fig:Amatrix} for a schematic of matrix $A$. Note that
the ensemble $\Ens(W,m_0,n_0)$ includes, as special case, rectangular
non-symmetric matrices with \iid entries. 
%
%
\begin{figure}[!t]
\centering
\includegraphics*[viewport = 60 60 700 520, width = 5in]{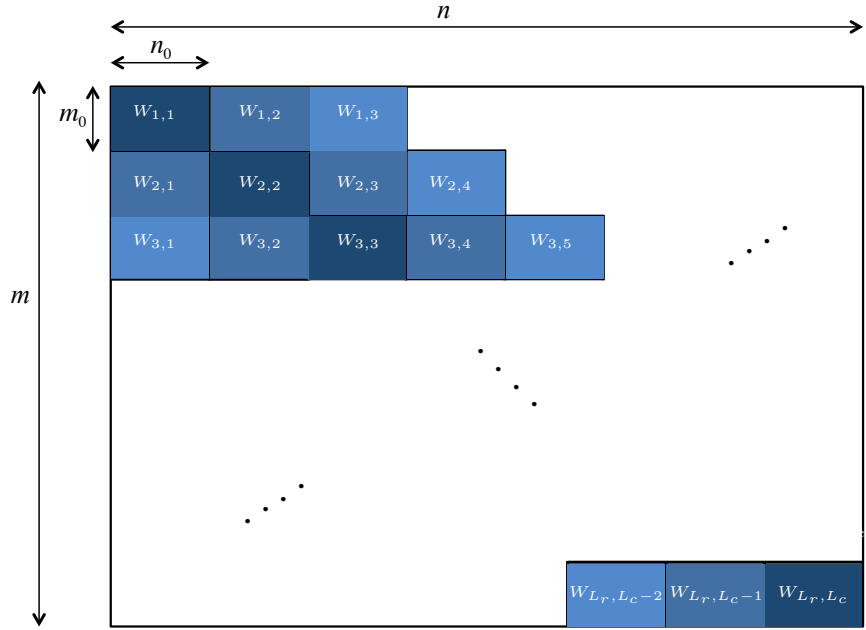}
\put(-292,217){\tiny{\white $W_{1,1}$}}
\put(-252,217){\tiny{\white $W_{1,2}$}}
\put(-215,217){\tiny{\white $W_{1,3}$}}
\put(-292,190){\tiny{\white $W_{2,1}$}}
\put(-252,190){\tiny{\white $W_{2,2}$}}
\put(-215,190){\tiny{\white $W_{2,3}$}}
\put(-180,190){\tiny{\white $W_{2,4}$}}
\put(-292,168){\tiny{\white $W_{3,1}$}}
\put(-252,168){\tiny{\white $W_{3,2}$}}
\put(-215,168){\tiny{\white $W_{3,3}$}}
\put(-180,168){\tiny{\white $W_{3,4}$}}
\put(-142,168){\tiny{\white $W_{3,5}$}}
\put(-145,60){\tiny{\white $W_{\Lr-1,\Lc-3}$}}
\put(-131,60){\tiny{\white $W_{\Lr-1,\Lc-2}$}}
\put(-91,60){\tiny{\white $W_{\Lr-1,\Lc-1}$}}
\put(-50,60){\tiny{\white $W_{\Lr-1,\Lc}$}}
\put(-127,35){\tiny{\white $W_{\Lr,\Lc-2}$}}
\put(-89,35){\tiny{\white $W_{\Lr,\Lc-1}$}}
\put(-48,35){\tiny{\white $W_{\Lr,\Lc}$}}
\caption{\small {Construction of the spatially coupled measurement matrix $A$ for compressive sensing as described in Section~\ref{sec:MatrixEnsemble}. The matrix is divided into blocks with size $m_0$ by $n_0$. (Number of blocks in each row and each column are respectively $\Lc$ and $\Lr$, hence $m= m_0 \Lr$, $n = n_0\Lc$). The matrix elements $A_{ij}$ are chosen as $\normal(0,\frac{1}{m_0}W_{\gr(i),\gr(j)})$. In this figure, $W_{i,j}$ depends only on $|i-j|$ and thus blocks on each diagonal have the same variance.}}  \label{fig:Amatrix}
\end{figure}
%
\subsection{AMP for compressed sensing reconstruction}

AMP algorithms were applied in \cite{DJM} to compressed sensing
reconstruction with spatially coupled sensing matrices \cite{KrzakalaEtAl}.
Here we follow the scheme and notations of \cite{DJM}. In particular, we assume that the unknown vector $x$ to be
reconstructed has entries whose empirical distribution converges weakly to a
probability measure $p_X$ over $\reals$.
The AMP algorithm takes the following form (initialized with $x_i^1 = \E_{p_X} (X)$ for all $i \in [n]$):
\begin{eqnarray}
x^{t+1} & = & \eta_t(x^t+(Q^t\odot A)^\sT r^t)\, ,\label{eq:AMP1}\\
r^t & = & y-Ax^t+\ons^t\odot r^{t-1}\, .\label{eq:AMP2}
\end{eqnarray}
Here, for each $t$, $\eta_t:\reals^n \to \reals^n$ is a differentiable non-linear function that
depends on the input distribution $p_X$. Further, for $v\in \reals^n$, we have
$\eta_t(v) = (\eta_{t,1}(v_1),\dotsc,\eta_{t,n}(v_n))$ for some functions $\eta_{i,t}: \reals \to \reals$.
The symbol $\odot$ indicates Hadamard (entrywise) product. The specific choices for $\eta_t, Q^t, \ons^t$ are given
in Section \ref{sec:GeneralAlgo} below.
  

\subsection{State evolution}
Given $W\in \reals_+^{\Rows \times \Cols}$ roughly row-stochastic, and 
undersampling rate $\delta \in (0,1)$, the corresponding state evolution is
defined as follows. 
Start with initial condition
\begin{eqnarray} \label{eqn:initial_cond}
\psi_i(0) = \infty  \mbox{ for all } i\in \Cols\, .
\end{eqnarray}
For all $t\ge 0$, $a\in \Rows$, and $i\in \Cols$, let
\begin{eqnarray}
\begin{split}\label{eq:ExplicitSE}
\phi_a(t) &=& \sigma^2+\frac{1}{\delta}\sum_{i\in\Cols } W_{a,i}\, \psi_i(t)\, ,\\
\psi_{i}(t+1) &= &\mmse\Big(\sum_{b\in\Rows}W_{b,i}\phi_b(t)^{-1}\Big)\,. 
\end{split}
\end{eqnarray}
Here and below, $\mmse(s)$ denotes the minimum mean square error 
in estimating $X\sim p_X$ from a noisy observation in Gaussian noise, at 
signal-to-noise ratio $s$. Formally,
\[
\mmse(s) = \E\{[X- \E[X|Y]]^2\}, \quad \quad Y = \sqrt{s} X+Z\,.
\]
%
%
%
\subsection{Construction of $\eta_t, \ons^t, Q^t$}
\label{sec:GeneralAlgo}
In the constructions for the matrix $Q^t$, the nonlinearities $\eta_t$, and the vector $\ons^t$,
we use the fact that the state evolution sequence can be precomputed.

Define $Q^t$ by
\begin{eqnarray}\label{eq:Q_def}
Q_{ij}^t\equiv\frac{\phi_{\gr(i)}(t)^{-1}}{\sum_{k=1}^{\Lr}W_{k,\gr(j)}\phi_{k}(t)^{-1}}\,.
\end{eqnarray}

The nonlinearity $\eta_t$ is chosen as follows:
\begin{eqnarray}\label{eq:eta_def1}
\eta_t(v) =
\big(\eta_{t,1}(v_1),\eta_{t,2}(v_2),\;\dots\;,\eta_{t,N}(v_N)\big)\, ,
\end{eqnarray}
where $\eta_{t,i}$ is the conditional expectation estimator
for $X\sim p_X$ in gaussian noise:
\begin{eqnarray}\label{eq:eta_def2}
\eta_{t,i}(v_i) = \E\big\{X\,\big|\, X+\, s_{\gr(i)}(t)^{-1/2}Z = v_i\,\big\}\,
,\;\;\;\; s_r(t) \equiv \sum_{u \in \Rows} W_{u,r}\phi_u(t)^{-1}\, .
\end{eqnarray}
Notice that the function  $\eta_{t,i}(\,\cdot\,)$ depends on $i$ only
through the group index $\gr(i)$, and in fact parametrically
through $s_{\gr(i)}(t)$. We  define $\teta_{t,i} = \eta_{t,u}$ for $i\in C_u$.

Finally, in order to define the vector $\ons^t_i$, let us introduce
the quantity (with $\eta'_{t,i}$ denoting the derivative of
$v_i\mapsto \eta_{t,i}(v_i)$)
\begin{eqnarray}
\<\eta'_{t}\>_u = \frac{1}{n_0}\sum_{i\in
  C_u}\eta'_{t,i}\big(x^t_i+((Q^t\odot A)^\sT r^t)_i\big)\, .
\end{eqnarray}
The vector $\ons^t$ is then defined by
\begin{eqnarray}
\label{eqn:ons_def}
\ons^t_i \equiv
\frac{1}{\delta}\sum_{u \in \Cols} W_{\gr(i),u}\tQ^{t-1}_{\gr(i),u}\,
\<\eta'_{t-1}\>_u\, ,\label{eq:OnsagerDef}
\end{eqnarray}
where we  defined $Q^t_{i,j} = \tQ^t_{r,u}$ for $i\in R_r$, $j\in
C_u$. 

The following Lemma (Lemma~$4.1$ in~\cite{DJM}) claims that the state evolution~\eqref{eq:ExplicitSE} allows an
exact asymptotic analysis of AMP algorithm~\eqref{eq:AMP1}-~\eqref{eq:AMP2} in the limit of a large number of
dimensions.

\begin{lemma}\label{lemma:SE}
Let $W\in\reals_+^{\Rows\times\Cols}$ be a roughly row-stochastic matrix and $\phi(t)$, $Q^t$, $\ons^t$ be defined as in Section
\ref{sec:GeneralAlgo}.
Let $m_0=m_0(n_0)$ be such that $m_0/n_0\to\delta$, as $n_0 \to \infty$, and let $A(n)\sim \Ens(W,m_0,n_0)$. 
Further suppose that the empirical distribution of the entries of $x(n)$ converges weakly to a probability measure $p_X$ on
$\reals$ with bounded second moment and the empirical second moment of $x(n)$ also converges to $\E_{p_X}(X^2)$. 
Similarly, suppose that the empirical distribution of the entries of $w(n)$ converges weakly to a probability measure $p_W$ on
$\reals$ with bounded second moment and the empirical second moment of $w(n)$ also converges to $\E_{p_W}(W^2)\equiv \sigma^2$. 
Then, for all $t\ge 1$, almost surely we have 
\begin{eqnarray}
\underset{n_0\to\infty}{\limsup}\frac{1}{n_0}\|x^t_{C_a}(A(n);y(n))-x_{C_a}\|_2^2=
\mmse\Big(\sum_{i\in \Rows}W_{i,a}\phi_i(t-1)^{-1} \Big)\, ,
\end{eqnarray}
for all $a \in \Cols$, where $x^t_{C_a}, x_{C_a} \in \reals^{n_0}$ respectively denote the restrictions of $x^t,x$ to indices in $C_a$.
\end{lemma}
%
\subsection{Proof of Lemma~\ref{lemma:SE}}
\label{sec:ProofLemmaSE}

We show that Lemma~\ref{lemma:SE} follows from Theorem~\ref{thm:SE}. Consider the following change of variables:
\begin{eqnarray}
\tx^{t+1} &=& x - (Q^t \odot A)^\sT r^t - x^t,\\
\tr^{t} &=& w - r^t.  
\end{eqnarray}
Rewriting Eqs~\eqref{eq:AMP1} and~\eqref{eq:AMP2} in terms of $\tx$ and $\tr$, we obtain
\begin{eqnarray}
\tx^{t+1} &=& (Q^t \odot A)^\sT (\tr^t-w) - \{\eta_{t-1}(x - \tx^t) - x\},\label{eqn:tx}\\
\tr^{t} &=& A \{\eta_{t-1}(x - \tx^t) - x\} + \ons^t \odot (\tr^{t-1} - w).\label{eqn:tr}
\end{eqnarray}
Let $q = \Lr+\Lc$ and define functions $e(\cdot,\cdot,\cdot;t), h(\cdot,\cdot,\cdot;t): \reals^q \times \reals^q \times [q] \to \reals^q$ as follows:
\begin{eqnarray*}
h(\vu,\vw,a;t) &=& \sqrt{\Lr}\, (\vu(a) - \vw(a))\, [\sqrt{W_{a,1}} \tQ^t_{a,1} , \dotsc, \sqrt{W_{a,\Lc}} \tQ^t_{a,\Lc}, \ast,\dotsc,\ast] \quad\;\;\;\text{ for }a\in[\Lr]\, ,\\
e(\vv,\vy, a;t) &=& \sqrt{\Lr}\, \{\teta_{t-1,a}(\vy(a) - \vv(a)) - \vy(a)\}\, [\sqrt{W_{1,a}} , \dotsc, \sqrt{W_{\Lr,a}}, \ast,\dotsc,\ast]\;\; \text{ for }a\in[\Lc]\,. 
\end{eqnarray*}
In our definition, we do not care about the values of entries represented by $\ast$, since they are irrelevant for our purposes. Values of $h(\vu,\vw,a;t)$ for $a \in\{\Lr+1,\dotsc,\Lr+\Lc\}$ and $e(\vv,\vy, a;t)$ for $a \in\{\Lc+1,\dotsc,\Lr+\Lc\}$ are also irrelevant for our purposes and can be defined arbitrarily. Note that $h,e \in \psL(2)$. We also define function $\he(\cdot,\cdot;t): \nSpace \times \nSpace \to \nSpace$ by letting $v' = \he(v,y;t)$ be given by $\vv'_j = e(\vv_j,\vy_j,\gr(j);t)$ for all $j \in [n]$. Similarly, $\hh(\cdot,\cdot;t): \mSpace \times \mSpace \to \mSpace$ is defined by letting $u' = \hh(u,w;t)$ be given by $\vu'_i = h(\vu_i,\vw_i,\gr(i);t)$ for all $i \in [m]$.
\smallskip
 
Let $\tA\in \reals^{m \times n}$ be a normalized version of $A$ obtained as in the following: 
\begin{eqnarray*}
\tA_{ij} = \sqrt{\frac{1}{\Lr\, W_{\gr(i),\gr(j)}}}\, A_{ij}.
\end{eqnarray*}
Therefore, $\tA$ has i.i.d. entries $\normal(0,1/m)$.
\begin{proposition}\label{pro:spatial-special}
Consider the following approximate message passing orbit with vectors $\{v^t,u^t\}_{t \ge 0}$, $v^t \in \nSpace$, $u^t \in \mSpace$:
\begin{eqnarray}
u^{t} &= & \tA\, \he(v^t,y;t) - \Ons_t \, \hh(u^{t-1},w;t-1)\, ,\label{eq:BipartiteAMP1}\\
v^{t+1} & = &\tA^{\sT}\, \hh(u^t,w;t) - \POns_t\,  \he(v^t,y;t)\, ,\label{eq:BipartiteAMP2}
\end{eqnarray}
for given $y\in \nSpace$ and $w \in \mSpace$. Here $\Ons_t: \mSpace\to\mSpace$ is the linear operator defined by
letting, for $z' = \Ons_t z$, and any $i\in [m]$,
\begin{eqnarray}
\vz'_i = \frac{1}{m}\left(\sum_{k\in [n]}\frac{\partial e}{\partial \vv}(\vv^t_k,\vy_k,\gr(k);t)\right) \vz_i\, .
\end{eqnarray}
Analogously $\POns_t: \nSpace\to\nSpace$ is the linear operator defined by
letting, for $z' = \POns_t z$, and any $j\in [n]$,
\begin{eqnarray}
\vz'_j = \frac{1}{m}\left(\sum_{l\in [m]}\frac{\partial h}{\partial \vu}(\vu^t_l,\vw_l,\gr(l);t)\right) \vz_j\, .
\end{eqnarray}
Assume that $y= (\vy_1,\dotsc, \vy_n)$, $ w = (\vw_1, \dotsc, \vw_m)$, and $v^1 = (\vv^1_1, \dotsc, \vv^1_n)$ are given by
\begin{align*}
\vy_k &= (\ast, \cdots, \ast, \underbrace{x_{k}}_{\rm{position }\,\, \gr(k)},\ast,\cdots,\ast) \in \reals^q,\quad \quad \forall k \in [n]\\
\vw_k &= (\ast, \cdots, \ast, \underbrace{w_k}_{\rm{position }\,\, \gr(k)}, \ast, \cdots, \ast) \in \reals^q,\quad \quad \forall k \in [m]\\
\vv^1_k &=  (\ast, \cdots, \ast, \underbrace{\tx^1_{k}}_{\rm{position }\,\, \gr(k)},\ast,\cdots,\ast)\in \reals^q,\quad \quad \forall k \in [n].
\end{align*}
Then, we have  $\vu^t_i(\gr(i)) = \tr^t_i$ and $\vv^{t+1}_j(\gr(j)) = \tx^{t+1}_j$, for all $i\in [m], j\in [n]$, and $t\ge 0$.
\end{proposition}  

We refer to Section~\ref{proof:spatial-special} for the proof of Proposition~\ref{pro:spatial-special}.

We proceed by constructing a suitable converging sequence of symmetric AMP instances, recognizing that a subset of the resulting orbit corresponds to the orbit $\{v^t,u^t\}$ of interest. The converging symmetric AMP instances $(A_s(N),g, x_s^0)$ are defined as: 

\begin{itemize}
\item The instances has dimensions $N = m+n$ and $q = \Lr + \Lc$.
\item Let $B_1 = C_1 + C_1^\sT$ and $B_2 = C_2 + C_2^\sT$, where $C_1\in \reals^{m\times m}$ and $C_2 \in \reals^{n \times n}$ have i.i.d. entries distributed as $\normal(0,(2m)^{-1})$. The symmetric matrix $A_s$ is given by
\begin{eqnarray*}
A_s = \sqrt{\frac{\delta}{\delta+1}}\begin{pmatrix}
B_1 & \tA\\
\tA^\sT &B_2
\end{pmatrix}.
\end{eqnarray*}
\item  Let $\vy_{s,i} = \vw_i\in \reals^q$ for $i \le m$ and $\vy_{s,i} = \vy_{i-m} \in \reals^q$ for $i > m$.
\item The initial condition is given by $x_s^0 = (\vx^0_{s,1},\cdots, \vx^0_{s,N}) \in \Space$, where $\vx^0_{s,i} = 0$ for $i \le m$ and 
$\vx^0_{s,i} = \vv^1_{i-m}$ for $m < i \le m+n$.
\item Finally, for any $\vx, \vy \in \reals^q$, $t \ge 0$, we let
\begin{eqnarray}
g(\vx,\vy,a,2t) &=& 0 \quad \quad\quad \quad\quad \quad\quad \quad\quad \quad  \quad \text{for } a\in \{1,\cdots,\Lr\},\label{eq:g1}\\
g(\vx,\vy,a,2t) &=& \sqrt{\tfrac{\delta+1}{\delta}}\,\, e(\vx,\vy,a-\Lr;t) \quad\;\text{ for } a\in \{\Lr+1,\cdots,\Lr+\Lc\},\label{eq:g2}\\
g(\vx,\vy,a,2t+1) &=& \sqrt{\tfrac{\delta+1}{\delta}}\,\, h(\vx,\vy,a;t+1) \quad \;\;\;\text{ for } a\in \{1,\cdots,\Lr\},\label{eq:g3}\\
g(\vx,\vy,a,2t+1) &=& 0 \quad \quad\quad \quad\quad \quad\quad\quad \quad\,  \quad \quad\text{for } a\in \{\Lr+1,\cdots,\Lr+\Lc\}.\label{eq:g4}
\end{eqnarray}
\end{itemize}

Now, it is easy to see that, for all $t \ge 0$,
\begin{align}
\vx^{2t+1}_{s,i} &= \vu^t_i, \quad \quad \text{for } i \le m,\label{eqn:sym1}\\
\vx^{2t}_{s,i} &= \vv^{t+1}_{i-m}, \quad \quad \text{for } m+1\le i \le m+n.\label{eqn:sym2}
\end{align}

Now we are ready to prove Lemma~\ref{lemma:SE} by applying Theorem~\ref{thm:SE}.

Fix $a' \in \{\Lr+1,\dotsc,\Lr+\Lc\}$ and $t \ge 1$. Let $a = a'-\Lr$ and choose function $\psi(\vx,\vy) = \{\teta_{t,a}(\vy(a) - \vx(a)) - \vy(a)\}^2$. Then, 
\begin{align}
\lim_{n_0\to \infty} \frac{1}{n_0} \sum_{j' \in C_{a'}} \psi(\vx^{2t}_{s,j}, \vy_{s,j}) 
&= \lim_{n_0\to \infty} \frac{1}{n_0} \sum_{j' \in C_{a'}} [\teta_{t,a}(\vy_{s,j'}(a) - \vx^{2t}_{s,j'}(a)) - \vy_{s,j'}(a)]^2\no\\
&\stackrel{(a)}{=} \lim_{n_0\to \infty} \frac{1}{n_0} \sum_{j \in C_a} [\teta_{t,a}(\vy_{j}(a) - \vv^{t+1}_{j}(a)) - \vy_{j}(a)]^2\no\\
&\stackrel{(b)}{=} \lim_{n_0\to \infty} \frac{1}{n_0} \sum_{j \in C_a} [\eta_{t,j}(x_{j} - \tx^{t+1}_j) - x_{j}]^2\no\\
&= \lim_{n_0\to \infty} \frac{1}{n_0} \sum_{j \in C_a} (x^{t+1}_j - x_{j})^2 = \lim_{n_0\to \infty} \frac{1}{n_0} \|x^{t+1}_{C_a} - x_{C_a}\|^2.\label{eqn:thm-to-lem}
\end{align}
Here $(a)$ follows from Eq.~\eqref{eqn:sym2} and the definition of $\vy_{s,j}$ (note that $j' = j -m$); $(b)$ follows from the fact $a = \gr(j)$ and Proposition~\ref{pro:spatial-special}.

Applying Theorem~\ref{thm:SE}, we have almost surely
\begin{eqnarray}\label{eqn:thmSE}
\lim_{n_0\to \infty} \frac{1}{n_0} \sum_{j' \in C_{a'}} \psi(\vx^{2t}_{s,j}, \vy_{s,j}) = \E[\eta_{t,a}(X+Z) - X]^2,
\end{eqnarray}
with $X\sim p_X$ and $Z\sim \normal(0,\Sigma^{2t}_{aa})$. Therefore, to complete the proof we need to show that
\begin{eqnarray}\label{eqn:claim_final}
(\Sigma^{2t}_{aa})^{-1} =  \sum_{i \in \Rows} W_{i,a} \phi_i(t)^{-1}.
\end{eqnarray}
Note that Eq.~\eqref{eq:GeneralSE1} reduces to:
\begin{eqnarray}
\Sigma^t = \frac{m_0}{m+n} \sum_{b'=1}^\Lr \hSigma_{b'}^{t-1} + \frac{n_0}{m+n} \sum_{b' = \Lr+1}^{\Lr+\Lc} \hSigma_{b'}^{t-1}.
\end{eqnarray}
By definition of function $g$ (see Eq.s~\eqref{eq:g1}-~\eqref{eq:g4}), it is easy to see that Eq.~\eqref{eq:GeneralSE2} reduces to:  
\begin{align}
(\hSigma^{2t}_{a'})_{ij} = 
\begin{cases}
0, \quad \quad \quad \quad \quad \quad& \text{for }a'\in[\Lr],\\
\frac{\delta+1}{\delta} \Lr \sqrt{W_{i,a} W_{j,a}}\, \E\{\eta_{t-1,a}(X - Z^t_a) - X\}^2,& \text{for }a'\in\{\Lr+1,\cdots,\Lr+\Lc\}, i,j \in [\Lr],\\
*,\quad \quad \quad \quad \quad \quad& \text{otherwise.}
\end{cases}
\end{align}
Here $a = a' - \Lr$, $X \sim p_X$ and $Z_a^t \sim \normal(0,\Sigma^{2t}_{aa})$. Also,
\begin{align}
(\hSigma^{2t-1}_{a'})_{ij} = 
\begin{cases}
\frac{\delta+1}{\delta} \Lr \sqrt{W_{a',i} W_{a',j}}\tQ^{t}_{a',i}\tQ^{t}_{a',j}\, \{\sigma^2 + \Sigma^{2t-1}_{a'a'}\},& \text{for }a'\in[\Lr], i,j \in [\Lc],\\
0, \quad \quad \quad \quad \quad \quad& \text{for }a'\in\{\Lr+1,\cdots,\Lr+\Lc\},\\
*,\quad \quad \quad \quad \quad \quad& \text{otherwise.}
\end{cases}
\end{align}
Consequently, we obtain
\begin{align*}
\Sigma^{2t}_{aa} & = \frac{m_0}{m+n} \sum_{b=1}^{\Lr} (\hSigma^{2t-1}_b)_{aa}\\
&=\frac{m_0\Lr}{m+n}\cdot \frac{\delta+1}{\delta} \sum_{b=1}^{\Lr} W_{b,a} (\tQ^t_{b,a})^2
\{\sigma^2 + \Sigma^{2t-1}_{bb} \}\\
&= \frac{m_0\Lr}{m+n}\cdot \frac{\delta+1}{\delta} \sum_{b=1}^{\Lr} W_{b,a} (\tQ^t_{b,a})^2
\{\sigma^2 + \f{n_0}{m+n} \sum_{c'=\Lr+1}^{\Lr+\Lc} (\hSigma^{2t-2}_{c'})_{bb} \}\\
&= \frac{m_0\Lr}{m+n}\cdot \frac{\delta+1}{\delta} \sum_{b=1}^{\Lr} W_{b,a} (\tQ^t_{b,a})^2
\bigg\{\sigma^2 + \frac{n_0 \Lr}{m+n}\cdot \frac{\delta+1}{\delta} \sum_{c = 1}^{\Lc} W_{b,c} \mmse((\Sigma^{2t-2}_{cc})^{-1})\bigg\}\\
& = \sum_{b=1}^{\Lr} W_{b,a} (\tQ^t_{b,a})^2 \bigg\{\sigma^2 + \frac{1}{\delta} \sum_{c = 1}^{\Lc} W_{b,c}\,\, \mmse((\Sigma^{2t-2}_{cc})^{-1})\bigg\}.
\end{align*}

We prove relation~\eqref{eqn:claim_final} using induction on $t$. The induction basis ($t=0$) is trivial. Suppose that the claim holds for $t-1$. Then,
\begin{align*}
\Sigma^{2t}_{aa}&= \sum_{b=1}^{\Lr} W_{b,a} (\tQ^t_{b,a})^2 \bigg\{\sigma^2 + \frac{1}{\delta} \sum_{c = 1}^{\Lc} W_{b,c}\,\, \mmse(\sum_{i\in \Rows} W_{i,c} \phi_i(t-1)^{-1})\bigg\}\\
& = \sum_{b=1}^{\Lr} W_{b,a} (\tQ^t_{b,a})^2 \phi_b(t)\\
&= \sum_{b=1}^{\Lr} W_{b,a} \frac{\phi_b(t)^{-2}}{\left(\sum_{k=1}^{\Lr} W_{k,a} \phi_k(t)^{-1}\right)^2}\phi_b(t)\\
&= \left( \sum_{b=1}^{\Lr} W_{b,a} \phi_b(t)^{-1} \right)^{-1}.
\end{align*}
This proves the induction claim for $t$. Combining~\eqref{eqn:thm-to-lem},\eqref{eqn:thmSE} and~\eqref{eqn:claim_final}, Lemma~\ref{lemma:SE} follows.

\subsubsection{Proof of Proposition~\ref{pro:spatial-special}}
\label{proof:spatial-special}
We prove the result by induction on $t$. For $t = 0$, the claim follows from our definition. Suppose that the claim holds for $t-1$, we prove that for $t$. 

Writing Eq.~\eqref{eq:BipartiteAMP1} for coordinate $i$, we have
\begin{eqnarray}
\vu_i^{t} = \sum_{k\in [n]} \tA_{ik} e(\vv^t_k,\vy_k,\gr(k);t) - 
\frac{1}{m}\left(\sum_{k\in [n]}\frac{\partial e}{\partial \vv}(\vv_k^t,\vy_k,\gr(k);t)\right) h(\vu^{t-1}_i, \vw_i,\gr(i); t-1)
\end{eqnarray}
Restricting to coordinate $\gr(i)$, we get
\begin{eqnarray}\label{eqn:u-coord}
\begin{split}
\vu_i^{t}(\gr(i)) = &\sum_{k\in [n]} \tA_{ik} [e(\vv^t_k,\vy_k,\gr(k);t)]_{\gr(i)}  \\
&-\frac{1}{m}\sum_{k\in [n]}[\frac{\partial e}{\partial v}(\vv_{k}^t(\gr(k)),\vy_{k}(\gr(k)),\gr(k);t)]_{\gr(i)}\,\, [h(\vu^{t-1}_i, \vw_i,\gr(i); t-1)]_{\gr(k)}.\\
\end{split}
\end{eqnarray}
Here, we have used the fact that $e(\vv^t_k,\vy_k, \gr(k),t)$ does not depend on $\vv^t_{k,l}$ for $l \neq \gr(k)$. 

Substituting for $e$ and $h$, we have
\begin{align}
\sum_{k\in [n]} \tA_{ik} [e(\vv^t_k,\vy_k,\gr(k);t)]_{\gr(i)} &= \sum_{k\in [n]} \tA_{ik} \sqrt{\Lr\, W_{\gr(i),\gr(k)}} 
\{\teta_{t-1,\gr(k)}(\vy_{k}(\gr(k)) - \vv^t_{k}(\gr(k))) -\vy_{k}(\gr(k))\} \nonumber\\
&= \sum_{k\in [n]} A_{ik} 
\{\eta_{t-1,k}(x_{k} - \tx^t_k) -x_{k}\},\label{eqn:term1-u}
\end{align}
where we used the induction hypothesis in the last step.
Furthermore, 
\begin{align}
&\frac{1}{m}\sum_{k\in [n]}[\frac{\partial e}{\partial v}(\vv_{k}^t(\gr(k)),\vy_{k}(\gr(k)),\gr(k);t)]_{\gr(i)}\,\, [h(\vu^{t-1}_i, \vw_i,\gr(i); t-1)]_{\gr(k)}\nonumber\\
&=-\frac{1}{m}\sum_{k\in [n]}\,\teta'_{t-1,\gr(k)}\left(\vy_{k}(\gr(k)) - \vv^t_{k}(\gr(k))\right) \sqrt{\Lr\,W_{\gr(i),\gr(k)}}
\left(\vu^{t-1}_{i}(\gr(i)) - \vw_{i}(\gr(i))\right) \sqrt{\Lr\,W_{\gr(i),\gr(k)}} \tQ^{t-1}_{\gr(i),\gr(k)}\nonumber\\
&= -\frac{1}{m}\sum_{k\in [n]}\, \Lr W_{\gr(i),\gr(k)}\tQ^{t-1}_{\gr(i),\gr(k)}\,\eta'_{t-1,k}(x_{k} - \tx^t_k) \,  (\tr^{t-1}_i - w_i)\nonumber\\
&= -\ons^t_i (\tr^{t-1}_i - w_i),\label{eqn:term2-u}
\end{align}
where we used the induction hypothesis in the second equality. The last equality follows from the definition of $\ons^t_i$ (see Eq.~\eqref{eqn:ons_def}); 


Using~\eqref{eqn:term1-u} and~\eqref{eqn:term2-u} in~\eqref{eqn:u-coord}, we obtain
\begin{eqnarray}
\vu_{i}^{t} (\gr(i))= \sum_{k\in [n]} A_{ik} 
\{\eta_{t-1,k}(x_{k} - \tx^t_k) -x_{k}\}
+\ons^t_i (\tr^{t-1}_i - w_i) = \tr^t_i,
\end{eqnarray}
where the second equality follows from~\eqref{eqn:tr}. This proves the induction claim for $\vu_{i}^{t}(\gr(i))$.

Next we prove the claim for $\vv^{t+1}_{j}(\gr(j))$. Writing Eq.~\eqref{eq:BipartiteAMP2} for coordinate $j$, we have
\begin{eqnarray}
\vv_j^{t+1} = \sum_{l\in [m]} \tA_{lj} h(\vu^t_l,\vw_l,\gr(l);t) - 
\frac{1}{m}\left(\sum_{l\in [m]}\frac{\partial h}{\partial \vu}(\vu^t_l,\vw_l,\gr(l);t)\right)
e(\vv^t_j,\vy_j,\gr(j);t)
\end{eqnarray}
Restricting to coordinate $\gr(j)$, we get
\begin{eqnarray}\label{eqn:v-coord}
\begin{split}
\vv_{j}^{t+1}(\gr(j)) =& \sum_{l\in [m]} \tA_{lj} [h(\vu^t_l,\vw_l,\gr(l);t)]_{\gr(j)} \\
&- \frac{1}{m} \sum_{l\in [m]} [\frac{\partial h}{\partial u}(\vu^t_{l}(\gr(l)),\vw_{l}(\gr(l)),\gr(l);t)]_{\gr(j)}
[e(\vv^t_j,\vy_j,\gr(j);t)]_{\gr(l)}.
\end{split}
\end{eqnarray}
Here, we have used the fact that $h(\vu^t_l,\vw_l, \gr(l),t)$ does not depend on $\vu^t_{l,k}$ for $k \neq \gr(l)$. 

Substituting for $e$ and $h$, we have
\begin{align}
\sum_{l\in [m]} \tA_{lj} [h(\vu^t_l,\vw_l,\gr(l);t)]_{\gr(j)} 
&= \sum_{l \in [m]} \tA_{lj} \sqrt{\Lr\, W_{\gr(l),\gr(j)}}\, \tQ^t_{\gr(l),\gr(j)}\,(\vu^t_{l}(\gr(l)) - \vw_{l}(\gr(l)))\nonumber\\
&= \sum_{l \in [m]} A_{lj} Q^t_{l,j}\,(\tr^t_l - w_l),\label{eqn:term1-v}
\end{align}
where in the last step we used the result $\vu^t_{l}(\gr(l)) = \tr^t_{l}$, proved above. Moreover,
\begin{align}
&\frac{1}{m}\sum_{l\in [m]} [\frac{\partial h}{\partial u}(\vu^t_{l}(\gr(l)),\vw_{l}(\gr(l)),\gr(l);t)]_{\gr(j)}
[e(\vv^t_j,\vy_j,\gr(j);t)]_{\gr(l)}\nonumber\\
&= \frac{1}{m} \sum_{l\in [m]} \sqrt{\Lr W_{\gr(l),\gr(j)}}\, \tQ^t_{\gr(l),\gr(j)} \{\teta_{t-1,\gr(j)}(\vy_{j}(\gr(j)) - \vv^t_{j}(\gr(j))) - \vy_{j}(\gr(j))\}\,
\sqrt{\Lr W_{\gr(l),\gr(j)}}\nonumber\\
& = \frac{1}{m}\left(\sum_{l\in [m]}\Lr W_{\gr(l),\gr(j)} Q^t_{l,j} \right) \{\eta_{t-1,j}(x_{j} - \tx^t_j) - x_{j}\}\nonumber\\
&=\eta_{t-1,j}(x_{j} - \tx^t_j) - x_{j}. \label{eqn:term2-v}
\end{align}

Using~\eqref{eqn:term1-v} and~\eqref{eqn:term2-v} in \eqref{eqn:v-coord}, we obtain
\begin{eqnarray}
\vv^{t+1}_{j}(\gr(j)) = \sum_{l \in [m]} A_{lj} Q^t_{l,j}\,(\tr^t_l - w_l) - \{\eta_{t-1,j}(x_{j} - \tx^t_j) - x_{j}\} = \tx^{t+1}_j,
\end{eqnarray}
where the second equality follows from~\eqref{eqn:tx}. This proves the induction claim for $\vv_{i}^{t+1}(\gr(i))$.

\section{Proof of Theorem~\ref{thm:SE}}
\label{sec:Proof}

\subsection{Definitions and notations}\label{subsec:definitions}
 Letting $m^t = f(x^t;t)$ for $t \ge 0$, Eq.~\eqref{eq:AMPGeneralDef}, becomes
\begin{eqnarray}
x^{t+1} = A\, m^t - \Ons_t \, m^{t-1}\, . \label{eq:AMPGeneralDef2}
\end{eqnarray}
This is initialized with $m^{-1} = 0$ and $m^0=m^{0,N}\in\Space$, a sequence of
deterministic vectors in $\Space$, with $\lim\sup_{N \to \infty} N^{-1}
\sum_{i=1}^N \|\vm^0_i\|^{2k-2} < \infty$. Also recall that the
vectors $y = (\vy_1,\dotsc,\vy_N) \in \Space$ are a fixed sequence
indexed by $N$, with converging empirical distributions. 

The idea of the proof is to study the stochastic process
$\{x^0,x^1,\dots,x^t,\dots\}$ taking values in $\Space$ without
conditioning on the matrix $A$. Instead, for each $t$, we will compute
the conditional distribution of $x^{t+1}$ given $x^0,\ldots,x^{t}$, and hence
$m^0,\ldots,m^{t}$.
More precisely, let $\sigal{S}_{t}$ be the $\sigma$-algebra generated by these variables.
We will compute the conditional distributions
$x^{t+1}|_{\sigal{S}_t}$, by characterizing the conditional distribution of the matrix $A$
given this filtration.

Throughout the proof, we identify $\Space$ with the set of matrices
$\reals^{N\times q}$. Adopting this convention, the linear operator
$\Ons_t$ can be more conveniently identified with the $q\times q$
matrix
\begin{eqnarray}
 \Om_t = \f{1}{N} \left( \sum_{j \in [N]} \f{\partial f^j}{\partial
     \vx}(\vx^t_j,t)\right)\, .
\end{eqnarray}
We therefore have $\Ons_tm_{t-1} = m_{t-1}\Om_t^{\sT}$ and
the equations for $x^1,\ldots,x^{t}$ can be written in matrix form as:
\begin{align}\label{eq:matrixYM}
\underbrace{\left[x^1|x^2+m^0\Om_1^{\sT}|\ldots|x^t+m^{t-2}\Om_{t-1}^{\sT}\right]}_{Y_{t-1}}&=A\underbrace{[m^0|\ldots|m^{t-1}]}_{M_{t-1}}.
\end{align}
In short $Y_{t-1}=A M_{t-1}$. Here and below we use
$[Q|P]$ to denote the matrix obtained by concatenating $Q$ and $P$
horizontally. 

We also introduce the notation $m_{\pl}^t$ for the projection of $m^t$
onto the column space of $M_{t-1}$.
More precisely, $m_{\pl}^t\in\reals^{N\times q}$ is the matrix whose
columns are the projections of the columns of $m^t$. This can be
written as 
\begin{align}
m_{\pl}^t = \sum_{i=0}^{t-1} m^i \alpha_i\,,\label{eq:alpha-defined}
\end{align}
where $\alpha_i\in\reals^{q\times q}$, $0\le i\le t-1$ contain the
coefficients of these projections.
Defining by $m_{\perp}^t=m^t-m_{\pl}^t$ the perpendicular component ,
we have $M_{t-1}^{\sT}m_{\perp}^t=0$.
We further denote by $\alpha \in \reals^{tq \times q}$ the matrix obtained by concatenating $\alpha_i$'s vertically. Using this notation, we have
\begin{eqnarray}
m_{\pl}^t = M_{t-1} \alpha\,.
\end{eqnarray}
%
%
For an integer $\ell \ge 1$, let $(\ell) = \{(\ell-1)q+1 , \dotsc, \ell q\}$. For a matrix $u$ and set of 
indices $I,J$, we let $u_{I,J}$ denote the submatrix formed by the rows in $I$ and columns in $J$.
We further let $u_{I}$ denote the submatrix containing just the rows in $I$.
For $v = (\vv_1, \dotsc,\vv_N)\in \Space$ and a set of indices $I = \{i_1,\dotsc,i_r\}$, let $v_I = (\vv_{i_1},\dotsc, \vv_{i_r})$.

Given $v\in \mSpace$ and $\varphi:\reals^q \to \reals^q$, we write $\varphi(v) = (\varphi(\vv_1),\dotsc,\varphi(\vv_m))$. 
We also define $\nphi(v) = [\f{\partial \varphi}{\partial \vv}(\vv_1), \dotsc, \f{\partial \varphi}{\partial \vv}(\vv_m)]^\sT$ with $\f{\partial \varphi}{\partial \vv} \in \reals^{q \times q}$ denoting the Jacobian matrix of $\varphi$. Note that $\nphi(v) \in \reals^{mq \times q}$. 
    
For $u \in \reals^{mq\times q}$, let $\<u\> = (1/m) \sum_{i=1}^m u_{(i)} \in \reals^{q\times q}$.
Also, for $u,v \in \Space$ we define 
\begin{eqnarray*}
\<u,v\> = \frac{1}{N} \sum_{i=1}^N \vu_i \vv_i^\sT \in \reals^{q \times q}.
\end{eqnarray*}
Note that $\<u,v\> = (1/N) u^\sT v $, as we regard $\Space \equiv \reals^{N\times q}$.

Given two random variables $X,Y$, and a $\sigma$-algebra
$\sigal{S}$, the notation
$X|_{\sigal{S}}\deq Y$ means that for
any integrable function $\phi$ and for any random variable $Z$ measurable
on $\sigal{S}$, $\E\{\phi(X)Z\}= \E\{\phi(Y)Z\}$.
In words we will say that $X$ is distributed as
(or is equal in distribution to) $Y$ \emph{conditional
on} $\sigal{S}$.
In case $\sigal{S}$ is the trivial $\sigma$-algebra we simply write
$X\deq Y$ (i.e. $X$ and $Y$ are equal in distribution).
For random variables $X,Y$ the notation $X\asequal Y$ means that $X$ and $Y$ are equal almost surely.

The large system limit will be denoted as
$\lim_{N\to\infty}$. In the large system limit,
we use the notation $\order_t(1)$ to
represent a matrix in $\reals^{tq \times q}$ (with $t$ fixed) such that all of its coordinates converge to 0 almost surely as $N\to\infty$.

The indicator function of property ${\cal A}$
is denoted by $\indicator({\cal A})$ or $\indicator_{{\cal A}}$. The normal distribution
with mean $\mu$ and variance $v^2$ is represented as $\normal(\mu,v^2)$.
%
%
\subsection{Main technical Lemma}

We will say that a convergent sequence of mappings $(\cF_N)_{N\in\naturals}$ is
non-trivial if there exists $\eps_0>0$ such that, for each $N$, $t\ge
0$, $a\in [q]$, $i\in [N]$,  $\vgamma\in\reals^q$ with
$\|\vgamma\|_2=1$, $b\in\reals$, we have
\begin{eqnarray*}
\int \big(\gamma^\sT g(\vx,\vy_i,a,t)-b\big)^2\de\vx\ge \eps_0\, .
\end{eqnarray*}
This condition is useful to rule out trivial degeneracies.
\begin{lemma}\label{lem:elephantSym}
Let $\{(A(N),\cF_N,x^{0,N})\}_N$ be a converging sequence of  AMP
instances as in
Theorem \ref{thm:SE} with $\cF_N$ non-trivial. Then the following hold
for all $t\in\naturals$
\begin{itemize}
\item[$(a)$]
\begin{align}
x^{t+1}|_{\sigal{S}_{t}}&\deq \sum_{i=0}^{t-1}x^{i+1}{\alpha_i}+ {\tA} m_\perp^t+\tM_{t-1}\order_{t-1}(1)\,,\label{eq:elephantSym-a}
\end{align}
where ${\tA}$ is an independent copy of $A$. The matrix $\tM_t$ is such that its columns form an orthogonal basis for the column space of $M_t$ and $\tM_t^\sT \tM_t=N\,\id_{tq\times tq}$.
Recall that, $\order_{t-1}(1)\in \reals^{(t-1)q \times q}$ is a random vector that converges to 0
almost surely as $N\to\infty$.

\item[$(b)$] For any pseudo-Lipschitz function $\phi:(\reals^q)^{t+2}\to\reals$ of order $k$,
\begin{align}
\lim_{N\to\infty}\f{1}{|C_a^N|}\sum_{i \in C_a^N} \phi(\vx_i^1,\ldots,\vx_i^{t+1}, \vy_i)
&\asequal\E\big[\phi( Z_a^1, \ldots, Z_a^{t+1}, Y_a)\big]\,. \label{eq:elephantSym-b}
\end{align}
where $(Z_a^1,\dots,Z_a^{t+1})$ is a Gaussian vector independent of
$Y_a \sim P_a$  and, for each $i$,
$Z_a^i\sim \normal(0,\Sigma^i)$

\item[$(c)$] For all $1\leq r,s\leq t, a \in [q]$ the following equations hold and all limits exist, are bounded and have degenerate distribution
(i.e. they are constant random variables):
\begin{align}
\lim_{N\to\infty}\< x_{C^N_a}^{r+1},x_{C^N_a}^{s+1}\>&\asequal\lim_{N\to\infty} \< m^{r},m^{s}\>\,. \label{eq:elephantSym-c}
\end{align}

\item[$(d)$] 
Consider any set of $q$ Lipschitz continuous functions $\varphi^a: \reals^q \times \reals^q \to \reals^q$. For all $1\leq r,s\leq t$, the following equations hold and all limits exist, are bounded and have degenerate distribution
(i.e. they are constant random matrices):
\begin{align}
\lim_{N\to\infty}\< x^{r+1}_{C^N_a}, \varphi(x^{s+1}_{C^N_a},y_{C^N_a})\>&\asequal
\lim_{N\to\infty}\< x^{r+1}_{C^N_a},x^{s+1}_{C^N_a}\> \< \nphi^a(x^{s+1}_{C^N_a},y_{C^N_a})\>\,. \label{eq:elephantSym-d1}
\end{align}

The Jacobians here are computed according to the first component. Define $\varphi : \Space \times \Space \to \Space$ by letting $v' = \varphi(u,v)$ be given by $\vv'_i = \varphi^a(\vu_i,\vv_i)$ for $i\in C^N_a$. Let $\nphi \in \reals^{Nq \times q}$ be the matrix obtained by concatenating the matrices $\nphi^a \in \reals^{|C^N_a|q\times q}$, for $a\in [q]$. Then, Eq.~\eqref{eq:elephantSym-d1} implies that for all $1\leq r,s\leq t$, the following equations hold:
\begin{align}
\lim_{N\to\infty}\< x^{r+1}, \varphi(x^{s+1},y)\>&\asequal \lim_{N\to\infty}\< x^{r+1},x^{s+1}\> \< \nphi(x^{s+1},y)\>\,. \label{eq:elephantSym-d}
\end{align}
\item[$(e)$] For $\ell=k-1$ and $a\in [q]$, the following holds almost surely
\begin{eqnarray}
\lim_{N\to\infty} \f{1}{|C^N_a|} \sum_{i\in C^N_a} \|\vx_i^{t+1}\|^{2\ell}<\infty\,. \label{eq:elephantSym-e}
\end{eqnarray}

\item[$(f)$]  For all $0\leq r\leq t$ the following limit exists and there are positive constants $\lbq_r$ (independent of $N$) such that almost surely
\begin{align}
\lim_{N\to\infty}\< m_\perp^r, m_\perp^r\> - \lbq_r\, \id_{q \times q} \succeq 0 \,. \label{eq:elephantSym-f}
\end{align}

\end{itemize}
\end{lemma}

%
%
\subsubsection{Proof of Theorem \ref{thm:SE}}

First assume that the sequence of functions $\cF_N$ is non-trivial.
Theorem~\ref{thm:SE} follows readily from
Lemma~\ref{lem:elephantSym}. More specifically, Theorem~\ref{thm:SE}
is obtained by applying Lemma~\ref{lem:elephantSym}$(b)$ to functions
$\phi(\vx^1_i,\dotsc,\vx^{t}_i) = \psi(\vx^t_i,\vy_i)$. 

Consider then the case in which $\cF_N$ is not non-trivial. In this
case we perturb the functions $g(\vx,\vy,a,t)$ as follows. Let $\varphi(\vx) : \reals^q \to \reals^q$
be a bounded smooth function. Define
\[
g^{\epsilon}(\vx,\vy,a,t) = g(\vx,\vy,a,t) + \epsilon \,\varphi(\vx).
\]
The resulting sequence of instances is then non-trivial and
state evolution applies. Call $\Sigma^t(\eps)$ the resulting
state evolution sequence, and denote by $x^t(\epsilon)$ the corresponding orbit.
Applying Theorem~\ref{thm:SE}, we have
\begin{eqnarray}\label{eqn:perturb1}
\lim_{N \to \infty} \frac{1}{N} \sum_{i=1}^N \psi(\vx^t_j(\epsilon), \vy_i) = \E\{\psi(Z^t_a(\epsilon), Y_a)\},
\end{eqnarray}
with $Z^t_a(\epsilon)\sim \normal(0,\Sigma^t(\epsilon))$. In order to prove the same theorem
for the orbit $\{x^t\}_{t\ge 0}$, we need to show the following two facts:
\begin{itemize}
\item[$(i)$] $\lim_{\epsilon \to 0} \E\{\psi(Z^t_a(\epsilon), Y_a)\} = \E\{\psi(Z^t_a,Y_a)\}$, with $Z^t_a \sim \normal(0,\Sigma^t)$.

\item[$(ii)$] Let $a_N(\epsilon) = \frac{1}{N} \sum_{i=1}^N \psi(\vx^t_i(\epsilon),\vy_i)$. Then $|a_N(\epsilon) - a_N(0)| \le C\epsilon$, with constant $C$ being independent of $N$.
\end{itemize}

Given $(i)$-$(ii)$, we have
\begin{align*}
\lim_{N\to \infty} \big|a_N(0) - \E\{\psi(Z^t_a,Y_a)\}\big| \le {\limsup}_{N \to \infty} \Big\{&\big|a_N(0) - a_N(\epsilon)\big| + \big|a_N(\epsilon) - \E\{\psi(Z^t_a(\epsilon),Y_a)\}\big|\Big\}\\
+& \big|\E\{\psi(Z^t_a(\epsilon),Y_a)\} - \E\{\psi(Z^t_a,Y_a)\}\big|\\
& \le C\varepsilon + 0 + \big|\E\{\psi(Z^t_a(\epsilon),Y_a)\} - \E\{\psi(Z^t_a,Y_a)\}\big|,
\end{align*}
where the last step follows from $(ii)$ and Eq.~\eqref{eqn:perturb1}. Therefore, taking the limit
of both sides as $\epsilon \to 0$,
\begin{align*}
\lim_{N\to \infty} \big|a_N(0) - \E\{\psi(Z^t_a,Y_a)\}\big| &\le 
\lim_{\epsilon \to 0} C\varepsilon + \lim_{\epsilon \to 0} \big|\E\{\psi(Z^t_a(\epsilon),Y_a)\} - \E\{\psi(Z^t_a,Y_a)\}\big| = 0\,,
\end{align*}
where the last step follows from $(i)$. This proves Theorem~\ref{thm:SE} for $\{x^t\}_{t\ge 0}$.

It remains to prove facts $(i)$-$(ii)$. The claim in $(i)$ follows readily by applying dominated convergence theorem and noting that $\psi(\cdot,\cdot)$ is Lipschitz continuous. 

To prove $(ii)$, write 
\begin{align*}
&|a_N(\epsilon) - a_N(0)| \\
&\le \frac{1}{N} \sum_{i=1}^N \big|\psi(\vx^t_i(\epsilon),\vy_i) - \psi(\vx^t_i,\vy_i) \big|\\
&\le \frac{L'}{N}\, \sum_{i=1}^N (1 + \|\vx^t_i(\epsilon)\|^{k-1} + \|\vx^t_i\|^{k-1} + \|\vy_i\|^{k-1})\, \|\vx^t_i(\epsilon) - \vx^t_i\|\\
&\le \frac{L'}{N}\, \Big\{\sum_{i=1}^N (1 + \|\vx^t_i(\epsilon)\|^{k-1} + \|\vx^t_i\|^{k-1} + \|\vy_i\|^{k-1})^2 \Big\}^{\frac{1}{2}}
\Big\{\sum_{i=1}^N \|\vx^t_i(\epsilon) - \vx^t_i\|^2\Big\}^{\frac{1}{2}}\\
&\le {3L'}\, \Big\{ 1 + \frac{1}{N} \sum_{i=1}^N \|\vx^t_i(\epsilon)\|^{2k-2} + \frac{1}{N}\sum_{i=1}^N \|\vx^t_i\|^{2k-2}
+ \frac{1}{N}\sum_{i=1}^N \|\vy_i\|^{2k-2}\Big\}^{\frac{1}{2}}  \Big\{\frac{1}{N}\sum_{i=1}^N \|\vx^t_i(\epsilon) - \vx^t_i\|^2\Big\}^{\frac{1}{2}}\,,
\end{align*}
where second inequality holds since $\psi \in \psL(k)$ and third inequality follows by using Cauchy-Schwartz inequality. In the last expression, the term in the first braces is bounded using the assumption on the second moment of $y$ and using part $(e)$ of Lemma~\ref{lem:elephantSym} for orbit $\{x^t(\epsilon)\}$.  To bound the second braces, note that both $A$ and $\Ons_t$ in the AMP iteration
(\ref{eq:AMPGeneralDef}) have bounded operator norm (the former with
probability $1-e^{-\Theta(N)}$). Since $g(\,\cdot\,,t)$ is Lipschitz
continuous and $\varphi(\vx)$ is bounded by assumption, we conclude that $\|x^t(\eps)-x^t\|^2\le
c^t N\eps^2$ for some absolute constant $c$. This completes the proof of fact $(ii)$.

\subsection{Proof of Lemma~\ref{lem:elephantSym}}
The proof is by induction on $t$. Let $\pty{B}_{t}$ be the property that  \eqref{eq:elephantSym-a}, \eqref{eq:elephantSym-b}, \eqref{eq:elephantSym-c}, \eqref{eq:elephantSym-d}, \eqref{eq:elephantSym-e}, and \eqref{eq:elephantSym-f} hold.

\subsubsection{Induction basis: $\pty{B}_0$}

Note that $x^1 = A m^0$.

\begin{itemize}
\item[$(a)$] $\sigal{S}_0$ is generated by $y$, $x^0$ and $m^0$.
Also $m^0=m^0_\perp$ since $M_{-1}$ is an empty matrix. Hence
\[
x^1|_{\sigal{S}_1}= Am^0_\perp.
\]

\item [$(b)$] Let $\phi:{\cal V}_{q,2}\to\reals$ be a pseudo-Lipschitz function of order $k$.  Hence, $|\phi(x)|\leq L(1+\|x\|^k)$.
Given $m^0$, $y$, the random variable  $\sum_{i\in C^N_a} \phi([Am^0]_i,\vy_i)/|C^N_a|$ is a sum of independent random variables. By Lemma \ref{lem:prop-Gaussian-matrix}(a) $[Am^0]_i\deq Z$ for $Z\sim\normal(0, \<m^0, m^0\>)$.
Using Eq.~\eqref{eq:InitialSE},
\begin{align*}
\lim_{N\to \infty}  \<m^0, m^0\> &=
\sum_{a \in [q]} c_a \left(\lim_{N \to \infty} \f{1}{|C_a^N|} \sum_{i \in C^N_a} (\vm^0_i)^\sT \vm^0_i\right)\\
&=  \sum_{a \in [q]} c_a \hSigma^0_a = \Sigma^1.
\end{align*}
Hence for all $p\geq 1$, there exists a constant $c_p$ such that $\E\{\|[Am^0]_i\|^{p}\}  \le \|\<m^0_\perp,m^0_\perp\>\|_2^{\f{p}{2}}\; \E_Z\|Z\|^p< c_p$, with $Z \sim \normal(0,\id_q)$. Next, we check conditions of Theorem \ref{thm:SLLN} for $X_{N,i}\equiv \phi(\vx_i^1,\vy_i)-\E_A\{\phi(\vx_i^1,\vy_i)\}$ for $\kappa > 0$,
\begin{align}
&\frac{1}{|C_a^N|}\sum_{i\in C^N_a} \E|X_{N,i}|^{2+\kappa}\\
&=\frac{1}{|C_a^N|}\sum_{i\in C^N_a}\E\big|\phi(\vx_i^1,\vy_i)-\E_A\{\phi(\vx_i^1,\vy_i)\}\big|^{2+\kappa}\nonumber\\
&=\frac{1}{|C_a^N|}\sum_{i\in C^N_a} \left|\E_{A,\tA}\left\{\phi([\tA m^0]_i,\vy_i)-\phi([A m^0]_i,\vy_i)\right\}\right|^{2+\kappa}\nonumber\\
&\le\frac{1}{|C_a^N|}\sum_{i\in C^N_a} \left|\E_{A,\tA}\left\{\phi([\tA m^0]_i,\vy_i)-\phi([A m^0]_i,\vy_i)\right\}\right|^{2+\kappa}\nonumber\\
&\le\frac{1}{|C_a^N|}\sum_{i\in C^N_a} \left|\E_{A,\tA}\left\{\|[\tA m^0]_i-[Am^0]_i\|(1+\|\vy_i\|^{k-1}+\|[\tA m^0]_i\|^{k-1}+\|[A m^0]_i\|^{k-1})\right\}\right|^{2+\kappa}\nonumber\\
&\le c+\frac{L'c'}{|C^N_a|}\sum_{i\in C^N_a} \|\vy_i\|^{(k-1)(2+\kappa)}\nonumber\\
&\le c+L'c'|C^N_a|^{\kappa/2}\left(\frac{1}{|C_a^N|}\sum_{i\in C^N_a} \|\vy_i\|^{2(k-1)}\right)^{1+\kappa/2} \le c'' |C_a^N|^{\kappa/2}.\nonumber
\end{align}
Here $\tA$ is an independent copy of $A$, and the last inequality uses assumption on empirical moments of $\{\vy_i\}_{i \in C^N_a}$.
By applying Theorem \ref{thm:SLLN}, we get
\[
\lim_{N\to\infty}\f{1}{|C_a^N|}\sum_{i \in C^N_a} \left[\phi(\vx_i^1,\vy_i)-\E_A\{\phi(\vx_i^1,\vy_i)\}\right]\asequal 0.
\]
Hence, using Lemma \ref{lem:SLLN4us} for $v=w$ and $\psi(\vy_i)=\E_Z\{
\phi(Z,\vy_i)\}$ we get
\[
\lim_{N\to\infty}\f{1}{C^N_a}\sum_{i\in C^N_a}\E_Z[\phi(\vx^1_i,\vy_i)]\asequal\E\big\{\phi(Z_a,Y_a)\big\},
\]
with $Z_a\sim \normal(0,\Sigma^1)$ independent of $Y_a \sim P_a$. Note that $\psi$ belongs to $\psL(k)$ since $\phi$ belongs to $\psL(k)$.

\item [$(c)$] Let $\hat{A} = A_{C^N_a}$ be the submatrix formed by the rows in $C^N_A$. Using Lemma \ref{lem:prop-Gaussian-matrix}(c), conditioned on $m^0$,
\[
\lim_{N\to\infty}\< x_{C^N_a}^1,x_{C^N_a}^1\>=\lim_{N\to\infty} \<\hat{A}m^0,\hat{A}m^0\>\asequal\lim_{N\to\infty}\< m^0,m^0\> = \Sigma^1\,.
\]

\item [$(d)$] Write
\[
\lim_{N\to \infty} \<x^1_{C^N_a},\varphi(x^1_{C^N_a},y_{C^N_a})\> =  \lim_{N\to \infty} \f{1}{|C^N_a|} \sum_{i \in C^N_a} \vx^1_i [\varphi^a(\vx^1_i,\vy_i)]^\sT 
\asequal \E(Z_a [\varphi^a(Z_a,Y_a)]^\sT),
\]
where the last step follows by applying $\pty{B}_0(b)$ to the functions $\phi(\vx^1_i,\vy_i) = \vx^1_i(l) [\varphi^a(\vx_i^1,\vy_i)]_k$, for all $l,k \in [q]$.
 Furthermore, using Lemma \ref{lem:stein}, 
 \[
 \E(Z_a [\varphi^a(Z_a,Y_a)]^\sT) = \Sigma^1\, \E([\f{\partial \varphi^a}{\partial \vz}(Z_a,Y_a)]^\sT)\,.
 \] 

As proved in part $(c)$, $\lim_{N \to \infty} \<x^1_{C^N_a},x^1_{C^N_a}\> = \Sigma^1$. Also, by part $(b)$, 
the empirical distribution of $\{(\vx_i^1,\vy_i)\}_{i \in C^N_a}$ converges weakly to the distribution
of $(Z_a,Y_a)$, and consequently we get 
\[
\lim_{N \to \infty} \<\nphi^a(x^1_{C^N_a},y_{C^N_a})\>  = 
\lim_{N\to\infty} \f{1}{|C^N_a|} \sum_{i\in C^N_a} [\f{\partial \varphi^a}{\partial \vx}(\vx_i^1,\vy_i)]^\sT \asequal
\E([\f{\partial \varphi^a}{\partial \vz}(Z_a, Y_a)]^\sT)\,.
\]
This proves Eq.~\eqref{eq:elephantSym-d1}. To prove Eq.~\eqref{eq:elephantSym-d}, notice that 
\begin{eqnarray}\label{eqn:tmp1}
\<x^1,\varphi(x^1,y)\>  
=\sum_{a\in [q]} c_a \<x^1_{C^N_a},\varphi(x^1_{C^N_a}, y_{C^N_a})\>\,.
\end{eqnarray}
Also,
\begin{eqnarray}\label{eqn:tmp2}
\lim_{N\to \infty} \<x^1,x^1\> = \sum_{a\in [q]} c_a  \lim_{N\to \infty} \<x^1_{C^N_a},x^1_{C^N_a}\>
 = \sum_{a\in [q]} c_a \Sigma^1 = \Sigma^1,
\end{eqnarray}
where the last step holds since $\sum_{a\in [q]} c_a = 1$. Further,
\begin{eqnarray}\label{eqn:tmp3}
\<\nphi(x^1,y)\> = \sum_{a\in [q]} c_a \<\nphi^a(x^1_{C^N_a},y_{C^N_a})\>
\end{eqnarray}
Combining Eqs.~\eqref{eqn:tmp1}, ~\eqref{eqn:tmp2},~\eqref{eqn:tmp3} and Eq.~\eqref{eq:elephantSym-d1}, we get the desired result.

\item[$(e)$] Similar to $(b)$, conditioning on $m^0$, the term $\sum_{i\in C^N_a} \|[Am^0]_i\|^{2\ell}/|C^N_a|$ is sum of independent random variables
and 
\[
\E\{\|[Am^0]_i\|^p\}
\le \|\<m^0_\perp,m^0_\perp\>^{\f{1}{2}}\|_2^p\, \E\{\|Z\|^{p}\} < c_p\,,
\]
for a constant $c_p$. Therefore,
by Theorem \ref{thm:SLLN}, we get
\[
\lim_{N\to\infty}\f{1}{|C^N_a|}\sum_{i\in C^N_a} \left[\|[Am^0]_i\|^{2\ell}-\E_A\{\|[Am^0]_i\|^{2\ell}\}\right]\asequal 0.
\]
But, $\f{1}{|C^N_a|}\sum_{i\in C^N_a} \E_A\{\|[Am^0]_i\|^{2\ell}\} \le \|\< m^0,m^0\>^{\f{1}{2}}\|_2^{\ell} \, \E_Z\{\|Z\|^{2\ell}\} <\infty$.

\item[$(f)$] Since $t=0$ and $m^0 = m^0_{\perp}$, the result follows from $\lim_{N \to \infty} \<m^0,m^0\> = \Sigma^1 $ and that 
$\Sigma^1 = \sum_{b\in [q]} c_b \hSigma^0_b \succ 0$.

\end{itemize}
\subsubsection{Proof of $\pty{B}_t$:}
Suppose that $\pty{B}_{t-1}$ holds. We prove $\pty{B}_t$.

\begin{itemize}
\item [$(f)$] It is sufficient to consider $r=t$. Write $m^t_{\perp} = m^t - m^t_{\pl}$ and recall that $m^t_{\pl} = \sum_{s=0}^{t-1} m^s \alpha_s$. Hence, for any $\gamma_0 \in \reals^q$, with $\|\gamma_0\| = 1$, we have
\[
\gamma_0^\sT \<m^t_{\perp},m^t_{\perp}\> \gamma_0 = 
\frac{1}{N} \sum_{i=1}^N  \left(\gamma_0^\sT \vm^t_i - \sum_{s=0}^{t-1} \gamma_0^\sT\alpha_s^\sT \vm^s_i \right)
\left(\gamma_0^\sT \vm^t_i - \sum_{s=0}^{t-1} \gamma_0^\sT \alpha_s^\sT \vm^s_i \right)^\sT\,.
\]
Note that the matrix 
\[
\alpha = (\alpha_0, \dotsc, \alpha_{t-1}) =   \Big[ \f{M_{t-1}^\sT M_{t-1}}{N}\Big]^{-1}   \f{M_{t-1}^\sT m^t}{N}\,,
\]
has a finite limit as $N \to \infty$ by the induction hypothesis
$\pty{B}_{t-1}(b)$. Furthermore, $\vm_i^t = g(\vx^t_i,\vy_i,a,t)$, for $i \in C^N_a$. By induction hypothesis $\pty{B}_{t-1}(a)$,
it is sufficient to show that
there exists $\rho>0$ depending on $t$ such that,
\begin{align}
\underset{N\to\infty}{\liminf}\frac{1}{N}\sum_{i=1}^N 
&\left(\gamma_0^\sT g(\vZ+\sum_{r=1}^{t-1}\alpha_{r-1}^{\sT}\vx^{r}_i,\vy_i,a,t) - \sum_{s=0}^{t-1} \gamma_0^\sT\alpha_s^\sT \vm^s_i \right)\cdot \nonumber\\
&\left(\gamma_0^\sT g(\vZ+\sum_{r=1}^{t-1}\alpha_{r-1}^{\sT}\vx^{r}_i,\vy_i,a,t) - \sum_{s=0}^{t-1} \gamma_0^\sT \alpha_s^\sT \vm^s_i \right)^\sT
\ge 2\rho\, ,\label{eq:PositiveRho}
\end{align}
where $\vZ = (\tA m_{\perp}^{t-1})^{\sT}e_i\in\reals^q$ ($e_i$ being the
$i$-the element of the canonical basis). By the strong law of large
numbers for triangular arrays, the above is lower bounded by  
\begin{align*}
\underset{N\to\infty}{\liminf\,}\frac{1}{N}\sum_{i=1}^N\, 
&\E_{\tA}\Big[\gamma_0^\sT g(\vZ+\sum_{r=1}^{t-1}\alpha_{r-1}^{\sT}\vx^{r}_i,\vy_i,a,t) - \sum_{s=0}^{t-1} \gamma_0^\sT\alpha_s^\sT \vm^s_i \Big]\cdot\\
&\E_{\tA}\Big[\gamma_0^\sT g(\vZ+\sum_{r=1}^{t-1}\alpha_{r-1}^{\sT}\vx^{r}_i,\vy_i,a,t) - \sum_{s=0}^{t-1} \gamma_0^\sT \alpha_s^\sT \vm^s_i \Big]^\sT\\
& \ge \underset{N\to\infty}{\liminf\,} \frac{1}{N}\sum_{i=1}^N 
\Var_{\vZ}\Big(\vgamma_0^{\sT}g(\vZ+\sum_{r=1}^{t-1}\alpha_{r-1}^{\sT}\vx^{r}_i,\vy_i,a,t)\Big)\, .
\end{align*}
The variance in the last expression is taken only with respect to
$\tA$ or, equivalently, with respect to $\vZ\sim\normal(0,
\<m^{t-1}_{\perp},m_{\perp}^{t-1}\>)$. Notice that the covariance of
$\vZ$ is lower bounded by $\rho'\,\id_{q\times q}$ for some $\rho'>0$,
by the induction hypothesis $\pty{B}_{t-1}(f)$. It is a
straightforward probability exercise to show that, for any
non-constant continuous
function $G:\reals^q\to\reals$, and any $U>0$ there exists $\eps>0$
such that 
\begin{eqnarray*}
\inf_{\|\va\|_2\le U}\Var_{\vZ}\big(G(\va+\vZ)\big)\ge \eps\, .
\end{eqnarray*}
Using $\pty{B}_{t-1}(e)$, we can choose $U$ large enough to ensure that there exists
at least $N/2$ values of the indices $i\in [N]$ such that
$\|\sum_{r=0}^{t-1}\alpha_{r-1}^{\sT}\vx^{r}_i\|\le U$. Note that $U$ and therefore $\eps$ depend on $t$ but do not depend on $N$. 
The lower bound (\ref{eq:PositiveRho}) follows then by taking $\rho = \eps/4$.


\item[$(a)$] Let $\Om_t \in \reals^{q\times q}$ be given by $\Om_t = \f{1}{N} \left( \sum_{j \in [N]} \f{\partial f^j}{\partial \vx}(\vx^t_j,t)\right)$. Further let $\matB$ be a square block-diagonal matrix of size $tq$ with matrices $\Om_0^\sT,\dotsc, \Om_{t-1}^\sT$ on its diagonal. Define $X_{t-1} = [x^1|x^2|\dotsc|x^t]$. Recalling the definition of $Y_{t-1}$ and $M_{t-1}$ from Section~\ref{subsec:definitions},
\[
Y_{t-1} = X_{t-1} + [0_{N\times q}|M_{t-2}] \matB\,.
\]
\begin{lemma}\label{lem:xt Conditioned}
The following holds
\[
x^{t+1}|_{\sigal{S}_{t}}\deq X_{t-1}(M_{t-1}^\sT M_{t-1})^{-1}M_{t-1}^\sT m^t_\pl+  P_{M_{t-1}}^\perp {\tA} m_{\perp}^t+M_{t-1}\order_{t-1}(1)\,.
\]
\end{lemma}
\begin{proof}
Lemma~$10$ in~\cite{BM-MPCS-2011} implies that $A|_{\sigal{S}_t} \deq E(A|_{\sigal{S}_t}) + \mathcal{P}_t(\tA)$, where $\tA \deq A$ is a random matrix independent of $\sigal{S}_t$ and $\mathcal{P}_t$ is the orthogonal projector onto subspace $V_t = \{A|AM_{t-1} = 0, A= A^\sT\}$. Following the same argument as in~\cite{BM-MPCS-2011}, we have
\begin{eqnarray*}
E(A|_{\sigal{S}_t}) &=& Y_{t-1}(M_{t-1}^\sT M_{t-1})^{-1}M_{t-1}^\sT + M_{t-1}(M_{t-1}^\sT M_{t-1})^{-1} Y_{t-1}^\sT\\
 &&- M_{t-1}(M_{t-1}^\sT M_{t-1})^{-1}M_{t-1}^\sT Y_{t-1} (M_{t-1}^\sT M_{t-1})^{-1}M_{t-1}^\sT\,.\\
 \mathcal{P}_t(\tA) &=&  P_{M_{t-1}}^\perp \tA P_{M_{t-1}}^\perp\,.
\end{eqnarray*}
Using $M_{t-1}^\sT m^t_{\perp} = 0$ and $Y_{t-1} = AM_{t-1}$, it is immediate to see that
\[
A|_{\sigal{S}_t} m^t \deq Y_{t-1}(M_{t-1}^\sT M_{t-1})^{-1}M_{t-1}^\sT m^t_{\pl}
+ M_{t-1}(M_{t-1}^\sT M_{t-1})^{-1} Y_{t-1}^\sT m^t_{\perp} + P_{M_{t-1}}^\perp \tA m^t_{\perp}\,.
\]
Moreover, $Y_{t-1}^\sT m^t_{\perp} = X_{t-1}^\sT m^t_{\perp}$ because $M_{t-2}^\sT m^t_{\perp} = 0$. Recalling $m^t_{\pl} = M_{t-1} \alpha$ we need to show
\begin{eqnarray}\label{eq:error(x)}
[0|M_{t-2}] \matB \alpha + M_{t-1} (M_{t-1}^\sT M_{t-1})^{-1} X_{t-1}^\sT m^t_{\perp} - m^{t-1} \Om_t^\sT = M_{t-1} \order_{t-1}(1)\,.
\end{eqnarray}
Note that we used the fact $\Ons_t m^{t-1} = m^{t-1} \Om_t^\sT$ which follows from our convention $\Space \equiv \reals^{N\times q}$.

Here is our strategy to prove \eqref{eq:error(x)}. The left hand side is a linear combination of $m^0,\ldots,m^{t-1}$. For any $\ell = 1,\ldots,t$ we will prove that the coefficient of
$m^{\ell-1}$ converges to $0$.  Note that the coefficients are matrices of size $q$. The coefficient of $m^{\ell-1}$  in the left hand side is equal to
\[\left[(M_{t-1}^\sT M_{t-1})^{-1}X_{t-1}^\sT m_\perp^t\right]_{(\ell)} - \Om_\ell^\sT (-\alpha_\ell)^{\indicator_{\ell\neq t}}=
\sum_{r=1}^{t} \left[(\f{M_{t-1}^\sT M_{t-1}}{N})^{-1}\right]_{(\ell),(r)} \< x^{r},m^t-\sum_{s=0}^{t-1} m^s \alpha_s\> - \Om_\ell^\sT (-\alpha_\ell)^{\indicator_{\ell\neq t}}\,.
\]
To simplify the notation denote the matrix $M_{t-1}^\sT M_{t-1}/N$ by $G$.
Therefore,
\[
 \lim_{N\to\infty}\textrm{Coefficient of } m^{\ell-1}=
\lim_{N\to\infty}\left\{\sum_{r=1}^{t} (G^{-1})_{(\ell),(r)} \< x^{r},m^t-\sum_{s=0}^{t-1} m^s \alpha_s\> - \Om_\ell^\sT (-\alpha_\ell)^{\indicator_{\ell\neq t}}\right\}.
\]
But using the induction hypothesis $\pty{B}_{t-1}(d)$ for $\varphi=f(\cdot;1),\dotsc, f(\cdot;t)$,
the term $\< x^{r},m^t-\sum_{s=0}^{t-1}m^s \alpha_s\> $ is almost surely equal to the
limit of
$\< x^{r},x^t\> \Om_t^\sT - \sum_{s=0}^{t-1} \< x^{r},x^s\> \Om_s^\sT \alpha_s$. This can be modified, using the induction hypothesis $\pty{B}_{t-1}(c)$, to
$\< m^{r-1},m^{t-1}\>\Om_t^\sT - \sum_{s=0}^{t-1}\< m^{r-1},m^{s-1}\> \Om_s^\sT \alpha_s$ almost surely, which can be written as $G_{(r),(t)}\Om_t^\sT - \sum_{s=0}^{t-1} G_{(r),(s)}\Om_s^\sT \alpha_s$.
Hence,
\begin{align*}
\lim_{N\to\infty}\textrm{Coefficient of } m^{\ell-1} &\asequal  \lim_{N\to\infty}\left\{\sum_{r=1}^{t} (G^{-1})_{(\ell),(r)} [G_{(r),(t)}\Om_t^\sT - \sum_{s=0}^{t-1}G_{(r),(s)}\Om_s^\sT \alpha_s]-\Om_\ell^\sT (-\alpha_\ell)^{\indicator_{\ell\neq t}}\right\}\\
&\asequal  \lim_{N\to\infty}\left\{\Om_t^\sT \indicator_{t=\ell}   - \sum_{s=0}^{t-1}\Om_s^\sT \alpha_s \indicator_{\ell=s} - \Om_\ell^\sT (-\alpha_\ell)^{\indicator_{\ell\neq t}}\right\}\\
&\asequal 0\,.
\end{align*}

Notice that in the above equalities we used the fact that $G$ has, almost surely, a non-singular limit as $N \to \infty$ which 
was discussed in part $(f)$.
\end{proof}

The proof of Eq.~\eqref{eq:elephantSym-a} follows immediately since 
the last lemma yields
\[
x^{t+1}|_{\sigal{S}_t} \deq \sum_{i=0}^{t-1} x^{i+1} \alpha_i + {\tA}m_\perp^{t}- M_{t-1}(M_{t-1}^\sT M_{t-1})^{-1}M_{t-1}^\sT \tA m_\perp^t+M_{t-1}\order_{t-1}(1)\,.
\]
Note that, using Lemma \ref{lem:prop-Gaussian-matrix}(d), as $N\to\infty$,
\begin{align*}
M_{t-1}(M_{t-1}^\sT M_{t-1})^{-1}M_{t-1}^\sT \tA m_\perp^t&\deq\tM_{t-1}\order_{t-1}(1)\,,
\end{align*}
which finishes the proof since $\tM_{t-1}\order_{t-1}(1)+M_{t-1}\order_{t-1}(1)=\tM_{t-1}\order_{t-1}(1)$.


\item[$(c)$]
For $r,s<t$ we can use induction hypothesis. For $r=t, s<t$,
\begin{align}
\< x_{C^N_a}^{t+1},x_{C^N_a}^{s+1}\>|_{\sigal{S}_t}&\deq \sum_{i=0}^{t-1}\alpha_i^\sT \< x_{C^N_a}^{i+1},x_{C^N_a}^{s+1}\> +
\< [P_{M_{t-1}}^{\perp} \tA m_\perp^{t}]_{C^N_a},x_{C^N_a}^{s+1}\> + \sum_{i=0}^{t-1} \order_1(1) \< m_{C^N_a}^i,x_{C^N_a}^{s+1}\>.\no
\end{align}
Now, by induction hypothesis $\pty{B}_{t-1}(d)$, for $\varphi(\vv,\vu)=g(\vv,\vu,a,i)$, each term $\< m_{C^N_a}^i,x_{C^N_a}^{s+1}\>$ has a finite limit. Thus,
\[
\lim_{N\to\infty}\sum_{i=0}^{t-1}\order_1(1)\< m_{C^N_a}^i,x_{C^N_a}^{s+1}\>\asequal0.
\]
We can use Lemma \ref{lem:prop-Gaussian-matrix} (b)-(d) for $\< [P_{M_{t-1}}^{\perp} \tA m_\perp^{t}]_{C^N_a},x_{C^N_a}^{s+1}\>$
to obtain $\< [P_{M_{t-1}}^{\perp} \tA m_\perp^{t}]_{C^N_a},x_{C^N_a}^{s+1}\> \to 0$, almost surely.
Finally, using induction hypothesis $\pty{B}_{s}(c)$ or $\pty{B}_{i}(c)$ for each term of the form $\< x_{C^N_a}^{i},x_{C^N_a}^{s+1}\>$ 
\begin{align}
\lim_{N\to\infty}\< x^{t+1}_{C^N_a},x^{s+1}_{C^N_a}\>&\asequal \lim_{N\to\infty}\sum_{i=0}^{t-1}\alpha_i^\sT\< m^{i},m^s\>\no\\
&\asequal\lim_{N\to\infty}\< m_\pl^t,m^s\>\asequal\lim_{N\to\infty}\< m^t,m^s\>\no\, ,
\end{align}
where the last line uses the definition of $\alpha_i$ and $m_\perp^t\perp m^s$.

For the case of $r=s=t$, we have
\begin{align}
\< x^{t+1}_{C^N_a},x^{t+1}_{C^N_a}\>|_{\sigal{S}_{t}}&\deq \sum_{i,j=0}^{t-1}\alpha_i^\sT\< x_{C^N_a}^{i+1},x_{C^N_a}^{j+1}\> \alpha_j+ 
\< [P_{M_{t-1}}^{\perp} \tA m_\perp^{t}]_{C^N_a},[P_{M_{t-1}}^{\perp} \tA m_\perp^{t}]_{C^N_a}\> + \order_1(1).\no
\end{align}
Note that the contribution of all products of the form$\< [P_{M_{t-1}}^{\perp} \tA m_\perp^{t}]_{C^N_a},x^{i+1}_{C^N_a}\>$ almost surely tend to $0$.
Now, using induction hypothesis $\pty{B}_i(c)$ and Lemma \ref{lem:prop-Gaussian-matrix} (c), we obtain
\begin{align}
\lim_{N\to\infty}\< x_{C^N_a}^{t+1},x_{C^N_a}^{t+1}\>|_{\sigal{S}_{t}}&\asequal \lim_{N\to\infty}\sum_{i,j=0}^{t-1}\alpha_i^\sT\< m^{i},m^j\> \alpha_j + \lim_{N\to\infty}\<m^t_\perp,m^t_\perp\>\no\\
&\asequal \lim_{N\to\infty}\< m_\pl^{t},m_\pl^t\> + \lim_{N\to\infty}\< m_\perp^{t},m_\perp^{t}\>\no\\
&\asequal \lim_{N\to\infty}\< m^{t},m^t\>\no.
\end{align}
%

\item[$(e)$] This part follows by a very similar argument to the one in the proof of Lemma~$1$ (Step $\pty{B}_t(e)$) in~\cite{BM-MPCS-2011}.

\item[$(b)$] Using part $(a)$ we can write
\[
\phi(\vx_i^1,\ldots,\vx_i^{t+1},\vy_i)|_{\sigal{S}_{t,t}}\deq\phi\left(\vx_i^1,\ldots,\vx_i^{t},\left[\sum_{r=0}^{t-1}x^{r+1} \alpha_r + {\tA}m_\perp^{t}+\tM_{t-1}\order_{t-1}(1)\right]_i,\vy_i\right).
\]
We show that we can drop the error term $\tM_{t-1}\order_{t-1}(1)$. Indeed, defining
\begin{eqnarray*}
a_{i}&=&\left(\vx_i^1,\ldots,\vx_i^{t},\left[\sum_{r=0}^{t-1}x^{r+1}\alpha_r + {\tA}m_\perp^{t}+\tM_{t-1}\order_{t-1}(1)\right]_i,\vy_i\right)\,,\\
b_{i}&=&\left(\vx_i^1,\ldots,\vx_i^{t},\left[\sum_{r=0}^{t-1} x^{r+1} \alpha_r + {\tA}m_\perp^{t}\right]_i,\vy_i\right)\,,
\end{eqnarray*}
by the pseudo-Lipschitz assumption
\[
|\phi(a_{i})-\phi(b_{i})|\leq L\,(1+\|a_{i}\|^{k-1} + \|b_{i}\|^{k-1})\left(\sum_{r=0}^{t-1}\|\tm_i^r\|\right)\, o(1).
\]
Therefore, using Cauchy-Schwartz inequality twice, we have
\begin{align}
&\f{1}{|C^N_a|}\Big|\sum_{i\in C^N_a}\phi(a_{i})-\sum_{i\in C^N_a}\phi(b_{i})\Big|\no\\
&\leq L'\bigg\{1 + \frac{1}{|C^N_a|} \sum_{i\in C^N_a}{\|a_{i}\|^{2k-2}} + \frac{1}{|C^N_a|} \sum_{i\in C^N_a}{\|b_{i}\|^{2k-2}}\bigg\}^{\f{1}{2}}  \bigg\{ \f{1}{|C^N_a|}\sum_{r=0}^{t-1}\|\tm^r\|^2 \bigg\}^{\f{1}{2}}  t^{\f{1}{2}} o(1)\,.
\label{eq:CS}
\end{align}
Also note that
\[
\f{1}{|C^N_a|}\sum_{i \in C^N_a} \|a_{i}\|^{2\ell}\leq (t+1)^{\ell}
\big\{\sum_{r=0}^{t}\f{1}{|C^N_a|}\sum_{i\in C^N_a} \|\vx_i^{r+1}\|^{2\ell}+\f{1}{|C^N_a|}\sum_{i\in C^N_a} \|\vy_i\|^{2\ell}\big\}\,,
\]
which is finite almost surely (for $\ell = k-1$) using $\pty{B}_{r}(e)$ for $ r\in [t]$
and the assumption on (the moment of) $y$. The term $|C^N_a|^{-1}\sum_{i \in C^N_a} \|b_{i}\|^{2\ell}$ is bounded almost surely since
\begin{align*}
\f{1}{|C^N_a|}\sum_{i \in C^N_a}\|b_{i}\|^{2\ell}&\leq
\f{C}{|C^N_a|}\sum_{i\in C^N_a}\|a_{i}\|^{2\ell}+C\sum_{r=0}^{t-1}
\f{1}{|C^N_a|}\sum_{i\in C^N_a}\|\tilde{\vm}_i^r\|^{2\ell} o(1)\\
& \leq
\f{C}{|C^N_a|}\sum_{i\in C^N_a} \|a_{i}\|^{2\ell}+C'\sum_{r=0}^{t-1}
\f{1}{|C^N_a|}\sum_{i\in C^N_a} \|{\vm}_i^r\|^{2\ell} o(1)\, ,
\end{align*}
where the last inequality follows from the fact that $[M_{t-1}^\sT M_{t-1}/N]$
has almost surely a non-singular limit as $N\to\infty$, as discussed in point $(f)$
above. Finally, for $r\le t-1$, each term
 $(1/|C^N_a|)\sum_{i\in C^N_a} \|{\vm}_i^r\|^{2\ell}$ can be easily proved to be bounded
using the induction hypothesis $\pty{B}_{t-1}(e)$.

Hence for any fixed $t$, \eqref{eq:CS} vanishes almost surely when $N$ goes to $\infty$.

Now given, $\vx^1,\ldots,\vx^t$,  consider the random variables
\[
\tilde{X}_{i}= \phi\left(\vx_i^1,\ldots,\vx_i^{t},\sum_{r=0}^{t-1} \alpha_r^\sT \vx^{r+1}_i + ({\tA}m_\perp^{t})_i,\vy_i\right)
\]
and $X_{i}\equiv\tilde{X}_{i}-\E_{{\tA}}\{\tilde{X}_{i}\}$.
Proceeding as in $\pty{B}_0$, and using the pseudo-Lipschitz property of
$\phi$, it is easy to check the conditions of Theorem \ref{thm:SLLN}.
We therefore get
\begin{multline}\label{eq:Ztb-halfway1}
\lim_{N\to\infty}\f{1}{|C^N_a|}\sum_{i\in C^N_a}\bigg[\phi\Big(\vx_i^1,\ldots,\vx_i^{t},\big[\sum_{r=0}^{t-1} x^{r+1}\alpha_r + {\tA}m_\perp^{t}\big]_i,\vy_i\Big)\\
-\E_{{\tA}}\Big\{\phi\Big(\vx_i^1,\ldots,\vx_i^{t},\big[\sum_{r=0}^{t-1}x^{r+1} \alpha_r + {\tA}m_\perp^{t}\big]_i,\vy_i\Big)\Big\}\bigg]\asequal 0.
\end{multline}
Note that $[{\tA}m_\perp^{t}]_i$ is a gaussian random vector with covariance $\<m_\perp^t, m_\perp^{t}\>$.
Further  $\<m_\perp^t, m_\perp^{t}\>$ converges to
a finite limit $\Gamma_t^2$ almost surely as $N\to\infty$.
Indeed  $\<m_\perp^{t}, m_\perp^{t}\> =  \<m^{t}, m^{t}\> - \<m_\pl^{t}, m_\pl^{t}\>$.
By $\pty{B}_t(c)$, $\<m^t,m^t\>$ converges to a finite limit.
Further, $\<m^t_{\pl},m^t_{\pl}\> =\sum_{r,s=0}^{t-1}\alpha_r^\sT\<x^r,x^s\>\alpha_s$
also converges since the products $\<x^r,x^s\>$ do and the
coefficients $\alpha_r$, $r\le t-1$ as discussed in $\pty{B}_t(f)$.

Hence we can use induction hypothesis $\pty{B}_{t-1}(b)$ for
\[
\widehat{\phi}(\vx_i^1,\ldots,\vx_i^{t},\vy_i)=\E_{Z}\Big\{\phi\Big(\vx_i^1,\ldots,\vx_i^{t},\sum_{r=0}^{t-1} \alpha_r^\sT \vx_i^{r+1}  +  \<m^t_\perp, m^t_\perp\>^{\f{1}{2}} Z,\vy_i\Big)\Big\}\,,
\]
with $Z \sim \normal(0,\id_{q\times q})$ independent of $\vx^{r+1}_i$, $r\le t-1$, to show
\begin{multline}\label{eq:Ztb-halfway2}
\lim_{N\to\infty}\f{1}{|C^N_a|}\sum_{i\in C^N_a} \E_{{\tA}}\left\{\phi\left(\vx_i^1,\ldots,\vx_i^{t},\left[\sum_{r=0}^{t-1}\alpha_r^\sT\vx_i^{r+1} + {\tA}m_\perp^{t}\right]_i,\vy_i\right)\right\}\\
\asequal\E\,\E_{Z}\Big\{\phi\Big(Z_a^1,\ldots Z_a^{t},\sum_{r=0}^{t-1}{\alpha_r^\sT} Z_a^{r+1} + \Gamma_t\, Z,Y_a\Big)\Big\}\, .
\end{multline}

Note that $\sum_{r=0}^{t-1}{\alpha_r^\sT} Z_a^{r+1} + \Gamma_t\, Z$ is a gaussian vector.
All that we need, is to show that the covariance matrix of this gaussian vector is $\Sigma^{t+1}$. But using a combination of \eqref{eq:Ztb-halfway1} and \eqref{eq:Ztb-halfway2}
for the pseudo-Lipschitz functions $\phi(\vv_1,\ldots,\vv_{t+1},\vy_i)= \vv_{t+1}(\ell) \vv_{t+1}(k)$, for all $\ell, k \in [q]$,
\begin{align}
\lim_{N\to\infty}\< x^{t+1}_{C^N_a},x^{t+1}_{C^N_a}\> \asequal\E\Big\{\Big(\sum_{r=0}^{t-1}{\alpha_r^\sT} Z_a^{r+1} + \Gamma_t\, Z\Big) \Big(\sum_{r=0}^{t-1}{\alpha_r^\sT} Z_a^{r+1} + \Gamma_t\, Z\Big) ^\sT\Big\}.\label{eq:var-temp}
\end{align}
On the other hand as proved in part $(c)$, 
\[
\lim_{N\to\infty}\< x^{t+1}_{C^N_a},x^{t+1}_{C^N_a}\>\asequal \lim_{N\to\infty}\< m^{t},m^{t}\> = \lim_{N\to\infty}\< f(x^t,t),f(x^t,t)\>\,.
\]
Hence,
\begin{align*}
\lim_{N\to\infty}\< x^{t+1}_{C^N_a},x^{t+1}_{C^N_a}\> &\asequal \lim_{N\to \infty} \f{1}{N} \sum_{i=1}^N f^i(\vx^t_i,t)[f^i(\vx^t_i,t)]^\sT\\
&= \sum_{a\in [q]} c_a \f{1}{|C^N_a|} \sum_{i \in C^N_a}  g(\vx^t_i,\vy_i,a,t) g(\vx^t_i,\vy_i,a,t)^\sT \,.
\end{align*}
By induction hypothesis $\pty{B}_{t-1}(b)$ for the pseudo-Lipschitz functions 
\[
\phi(\vv_1,\ldots,\vv_{t},\vy_i)=[g(\vv_{t},\vy_{i},a,t)]_{\ell} [g(\vv_{t},\vy_{i},a,t)]_{k}\,,
\]
for all $\ell,k\in [q]$, we get 
\[
 \f{1}{|C^N_a|} \sum_{i \in C^N_a} g(\vx^t_i,\vy_i,a,t) g(\vx^t_i,\vy_i,a,t)^\sT \asequal
\E\Big\{g(Z_a^t,Y_a,a,t) g(Z_a^t,Y_a,a,t)^\sT\Big\}  = \hSigma_a^t\,.
\]
Consequently,
\[
\lim_{N\to\infty}\< x^{t+1}_{C^N_a},x^{t+1}_{C^N_a}\>\asequal  \sum_{a\in [q]} c_a \hSigma_a^t = \Sigma^{t+1}\,.
\]
which proves the claim.

\item[$(d)$] In a very similar manner to the proof of $\pty{B}_0(d)$, using part $(b)$ for
the pseudo-Lipschitz function $\phi:{\cal V}_{q,t+2}\to\reals$ given by
$\phi(\vx_i^1,\ldots,\vx_i^{t+1},\vy_i)= \vx_i^{r+1}(l)[\varphi(\vx_i^{s+1},\vy_i)]_k$, for all $l,k \in [q]$, we can obtain
\[
\lim_{N\to\infty}\< x^{r+1}_{C^N_a},\varphi(x^{s+1}_{C^N_a},y_{C^N_a})\>\asequal \E(Z_a^{r+1} [\varphi( Z_a^{s+1},Y_a)]^ \sT )\,,
\]
for gaussian vectors $Z_a^{r+1}~\sim \normal(0,\Sigma^{r+1})$, $Z_a^{s+1}~\sim \normal(0,\Sigma^{s+1})$. Using Lemma \ref{lem:stein}, we have almost surely,
\[
\E(Z_a^{r+1} [\varphi( Z_a^{s+1},Y_a)]^\sT ) = \cov(Z_a^{r+1},Z_a^{s+1})\,\,\E([\f{\partial \varphi^a}{\partial \vz}(Z_a^{s+1},Y_a)]^\sT)\,.
\]
By another application of part $(b)$ for $\phi(\vx_i^1,\ldots,\vx_i^{t+1},\vy_i)=\vx_i^{r+1}(l)\vx_i^{s+1}(k)$ for all $l,k\in [q]$, 
\[
 \lim_{N\to\infty}\< x^{r+1}_{C^N_a},x^{s+1}_{C^N_a}\> = \cov(Z_a^{r+1},Z_a^{s+1})\,.
\]
Similar to $\pty{B}_0(d)$ we also have $\lim_{N\to\infty}\< \nphi^a(x^{s+1}_{C^N_a},y_{C^N_a})\> = \E([\f{\partial\varphi^a}{\partial \vz}(Z_a^{s+1},Y_a)]^\sT)$, almost surely, as the empirical distribution of $\{(\vx^{s+1}_i,\vy_i)\}_{i \in C^N_a}$ converges weakly to the distribution of $(Z_a^{s+1},Y_a)$. This finishes the proof of Eq.~\eqref{eq:elephantSym-d1}.

Eq.~\eqref{eq:elephantSym-d} follows from Eq.~\eqref{eq:elephantSym-d1} exactly by the same argument as in $\pty{B}_0(d)$.

\end{itemize}
%
%
\appendix
\section{Reference probability results}\label{app:ProbFacts}
In this appendix, we summarize a few probability facts that are repeatedly used in the proof of Lemma~\ref{lem:elephantSym}.
We start by the following strong law of large numbers (SLLN) for triangular arrays of independent but not identically distributed random variables. The form stated below follows immediately from~\cite[Theorem 2.1]{SLLN2}.
\begin{theorem}[SLLN, \cite{SLLN2}]\label{thm:SLLN}
Let $\{X_{n,i}: 1\le i\le n, n\ge 1\}$ be a triangular array of random
variables with $(X_{n,1},\dots,X_{n,n})$ mutually independent
with mean equal to zero for each $n$ and $n^{-1}\sum_{i=1}^n\E|X_{n,i}|^{2+\kappa}\le cn^{\kappa/2}$ for some $0<\kappa<1$, $c<\infty$.
Then $\frac{1}{n}\sum_{i=1}^nX_{i,n}\to 0$ almost surely for $n\to\infty$.
\end{theorem}

Next, we present a standard property of Gaussian matrices without proof. This is a generalization of~\cite[Lemma 2]{BM-MPCS-2011}.
\begin{lemma}\label{lem:prop-Gaussian-matrix} For any deterministic $u \in \Space$, $v \in \nSpace$ and a gaussian matrix ${\tA} \in \reals^{n\times N}$  with i.i.d. entries $\normal(0,1/N)$, we have
\begin{itemize}
\item[(a)] $[\tA u]_i \deq \<u,u\> ^{\frac{1}{2}} \vz$, where $\vz \sim \normal(0,\id_{q\times q})$.

\item [(b)] $\<\tilde{A}u,v\> \deq \<u,u\>^{\frac{1}{2}} \<v,z\>$, where $z\in \nSpace$, $\vz_i \sim \normal(0,\id_{q\times q})$.

\item[(c)] $\lim_{n\to\infty} \<{\tA}u,{\tA}u\> = \<u,u\>$ almost surely.

\item[(d)] Consider, for $d\le n$, a $d$-dimensional subspace $W$ of
$\reals^n$, an orthogonal basis $w_1,\ldots,w_d$ of $W$ with $\|w_i\|^2= n$ for $i=1,\ldots,d$, and the orthogonal projection $P_W$ onto $W$. Then for $D = [w_1|\dotsc|w_d]$, and $u \in \Space$ with $\<u,u\> = \id_{q\times q}$, we have $P_W \tilde{A} u\deq D x$ where $x \in {\cal V}_{q,d}$ satisfies: $\lim_{n\to\infty}\|x\|\asequal 0$.
(the limit being taken with $d$ fixed).
\end{itemize}
\end{lemma}
%
\begin{lemma}[Stein's Lemma \cite{Ste72}]\label{lem:stein}
For jointly gaussian random vectors $Z_1,Z_2 \in \reals^q$ with zero mean, and any function $\varphi:\reals^q\to\reals^{\tq}$ where $\E\{\f{\partial \varphi}{\partial \vz}(Z_1)\}$ and
$\E\{Z_1 [\varphi(Z_2)]^\sT\}$ exist, the following holds
\[
\E\{Z_1[\varphi(Z_2)]^\sT\}=\cov(Z_1,Z_2)\, \E\{[\f{\partial \varphi}{\partial \vz}(Z_2)]^\sT\}\,.
\]
\end{lemma}
The following law of large numbers is a generalization of ~\cite[Lemma 4]{BM-MPCS-2011}
and can be proved in a very similar manner.
%
\begin{lemma}\label{lem:SLLN4us}
Let $k\geq 2$ and consider a sequence of vectors $\{v(N)\}_{N\ge 0}$ in $\Space$,
 whose
empirical distribution, denoted by $\empr_{v(N)}$, converges weakly to a probability measure
$p_{V}$ on $\reals^q$, such that $\E_{p_V}(\|V\|^k) < \infty$. Further
assume $\E_{\empr_{v(N)}}(\|V\|^k)\to\E_{p_{V}}(\|V\|^k)$ as $N\to\infty$.
Then, for any pseudo-Lipschitz function $\psi:\reals^q\to\reals$ of order $k$:
\begin{eqnarray}
\lim_{N\to\infty}\f{1}{N}\sum_{i=1}^N\psi(\vv_{i}) \asequal\E\big[ \psi(V)\big]\,.\label{eq:empirical->p_X}
\end{eqnarray}
\end{lemma}

\bibliographystyle{amsalpha}

\bibliography{all-bibliography}

\end{document}